\documentclass[review]{article}
\usepackage{mathrsfs}
\usepackage{amsmath,amssymb,amsthm}
\newtheorem{theorem}{Theorem}[section]
\newtheorem{lemma}[theorem]{Lemma}
\newtheorem{corollary}[theorem]{Corollary}
\newtheorem{definition}[theorem]{Definition}
\newtheorem{remark}[theorem]{Remark}
\usepackage[english]{babel}
\numberwithin{equation}{section}
\title{A note on the $H^{s}$-critical inhomogeneous nonlinear Schr\"{o}dinger equation}
\author
{{\bf
JinMyong An,
JinMyong Kim
\footnote{Corresponding author: JinMyong Kim (jm.kim0211@ryongnamsan.edu.kp)}
}\\
\footnotesize{Faculty of Mathematics, {\bf Kim Il Sung} University, Pyongyang, Democratic People's Republic of Korea}\\%address
}

\date{}
\begin{document}
\maketitle
\begin{abstract}
In this paper, we consider the Cauchy problem for the $H^{s}$-critical
inhomogeneous nonlinear Schr\"{o}dinger (INLS) equation
\[iu_{t} +\Delta u=\lambda |x|^{-b} f(u),\;
u(0)=u_{0} \in H^{s} (\mathbb R^{n}),\]
where $n\in \mathbb N$, $0\le s<\frac{n}{2}$, $0<b<\min \left\{2,\;n-s,\; 1+\frac{n-2s}{2} \right\}$ and $f(u)$ is a nonlinear function that behaves like $\lambda |u|^{\sigma } u$ with $\lambda \in \mathbb C$ and $\sigma=\frac{4-2b}{n-2s}$.
First, we establish the  local well-posedness as well as the small data global well-posedness  in $H^{s}(\mathbb R^{n})$ for the $H^{s}$-critical INLS equation by using the contraction mapping principle based on the Strichartz estimates in Sobolev-Lorentz spaces.
Next, we obtain some standard continuous dependence results for the $H^{s}$-critical INLS equation.
Our results about the well-posedness and standard continuous dependence for the $H^{s}$-critical INLS equation improve the ones of Aloui-Tayachi [Discrete Contin. Dyn. Syst. 41 (11) (2021), 5409-5437] by extending the validity of $s$ and $b$.
Based on the local well-posedness in $H^{1}(\mathbb R^{n})$, we finally establish the blow-up criteria for $H^{1}$-solutions to the focusing energy-critical INLS equation. In particular, we prove the finite time blow-up for finite-variance, radially symmetric or cylindrically symmetric initial data.
\end{abstract}

\noindent {\bf Keywords}: Inhomogeneous nonlinear Schr\"{o}dinger equation; $H^{s}$-critical; Well-posedness; Blow-up\\

\noindent {\bf Mathematics Subject Classification (2020)}: 35Q55, 35A01, 35B30, 35B44, 46E35
%%%%%%%%%%%%%%%%%%%%%%%%%%%%%%%%%%%%%%%%%%%%%%%%%%%%%%%%%%%%%%%%%%%%%%%%%%%%%
\section{Introduction}

In this paper, we study the Cauchy problem for the inhomogeneous nonlinear Schr\"{o}dinger (INLS) equation
\begin{equation} \label{GrindEQ__1_1_}
\left\{\begin{array}{l} {iu_{t} +\Delta u=|x|^{-b} f{(u),}} \\ {u(0,\; x)=u_{0} (x),} \end{array}\right.
\end{equation}
where $u_{0} \in H^{s} (\mathbb R^{n})$,~$0\le s<\frac{n}{2}$, $0<b<\min \left\{2,\; n-s, \; 1+\frac{n-2s}{2}\right\}$ and $f$ is of class $X\left(\sigma,s,b\right)$ (see Definition \ref{defn 1.1.}).
\begin{definition}[\cite{AK211}]\label{defn 1.1.}
\textnormal{Let $f:\mathbb C\to \mathbb C$, $s\ge 0$, $\sigma >0$, $0\le b<2$ and let $\left\lceil s\right\rceil $ denote the minimal integer which is larger than or equal to $s$. For $k\in \mathbb N$, let the $k$-th order derivative of $f(z)$ be defined under the identification $\mathbb C=\mathbb R^{2}$ $($see Section 2$)$. Let us define
\begin{equation} \label{GrindEQ__1_2_}
\sigma _{s} =\left\{\begin{array}{l} {\frac{4-2b}{n-2s} ,\;0\le s<\frac{n}{2} ,} \\ {\infty ,\; s\ge \frac{n}{2} .} \end{array}\right.
\end{equation}
We say that $f$ is of class $X\left(\sigma,s,b\right)$ if it satisfies one of the following conditions:
\begin{itemize}
  \item $f(z)$ is a polynomial in $z$ and $\bar{z}$ satisfying $1<\deg (f)=1+\sigma\le 1+\sigma_{s}$.
  \item $\left\lceil s\right\rceil-1 < \sigma\le \sigma_{s}$ and $f\in C^{\max\left\{\left\lceil s\right\rceil, 1 \right\}} \left(\mathbb C\to \mathbb C\right)$ satisfies
      \begin{equation} \label{GrindEQ__1_3_}
        \left|f^{(k)}(z)\right|\lesssim|z|^{\sigma +1-k},
      \end{equation}
      for any $0\le k\le \max\left\{\left\lceil s\right\rceil, 1 \right\}$ and $z\in \mathbb C$.
\end{itemize}}
\end{definition}
\begin{remark}[\cite{AK211}]\label{rem 1.2.}
\textnormal{Let $s\ge 0$, $0<b<2$ and $0<\sigma \le \sigma_{s}$. If $\sigma $ is not an even integer, assume that $\sigma >\left\lceil s\right\rceil-1 $. Then we can easily verify that $f(u)=\lambda |u|^{\sigma } u$ with $\lambda \in \mathbb C$ is a model case of class $X\left(\sigma,s,b\right)$.}
\end{remark}
The case $b=0$ is the well-known classic nonlinear Schr\"{o}dinger (NLS) equation
which has been widely studied during the last three decades. See, for example, [2--7] and the references therein. On the other
hand, in the end of the last century, it was suggested that in some situations the
beam propagation can be modeled by the inhomogeneous nonlinear Schr\"{o}dinger
equation in the following form:
\begin{equation} \label{GrindEQ__1_4_}
iu_{t} +\Delta u=V\left(x\right)|u|^{\sigma } u.
\end{equation}
The potential $V(x)$ accounts for the inhomogeneity of the medium. We refer the reader to \cite{G00, LT94} for the physical background of \eqref{GrindEQ__1_4_}. The particular case $V(x)=|x|^{-b}$ appears naturally as a limiting case of potentials $V(x)$ that decays as $|x|^{-b}$ at infinity (see \cite{GS08}).

Before recalling the known results for \eqref{GrindEQ__1_1_}, let us recall some facts for this equation. The INLS equation \eqref{GrindEQ__1_1_} has
the following equivalent form:
\begin{equation} \label{GrindEQ__1_5_}
u(t)=S(t)u_{0} -i\lambda\int _{0}^{t}S\left(t-\tau
\right)|x|^{-b} f(u(\tau))d\tau  ,
\end{equation}
where $S(t)=e^{it\Delta } $ is the Schr\"{o}dinger semi-group. When $0\le s<\frac{n}{2} $, $\sigma _{s} $ given in \eqref{GrindEQ__1_2_} is said to be a $H^{s}$-critical power. If $s\ge 0$, $\sigma (<\sigma _{s})$ is said to be a $H^{s}$-subcritical power. See \cite{AK211} for example.
Especially, if $\sigma =\frac{4-2b}{n} $, the problem is known as $L^{2}$-critical or mass-critical. If $\sigma =\frac{4-2b}{n-2}$ with $n\ge3$, it is called $H^{1}$-critical or energy-critical. If $f(u)=\lambda |u|^{\sigma } u$ with $\lambda\in \mathbb R$, then the INLS equation \eqref{GrindEQ__1_1_} has the conserved mass and energy, defined respectively by
\begin{equation} \label{GrindEQ__1_6_}
M\left(u(t)\right):=\left\| u(t)\right\| _{L^{2}
(\mathbb R^{n})}^{2}=M\left(u_{0} \right),
\end{equation}
\begin{equation} \label{GrindEQ__1_7_}
E\left(u(t)\right):=\frac{1}{2} \left\|
 u(t)\right\| _{\dot{H}^{1} (\mathbb R^{n})}^{2} +\frac{\lambda
}{\sigma+2}\int_{\mathbb R^{n}}{|x|^{-b}\left|u\right|^{\sigma+2}dx}=E\left(u_{0} \right).
\end{equation}

The local and global well-posedness as well as blow-up and scattering in the energy space $H^{1} (\mathbb R^{n} )$ for \eqref{GrindEQ__1_1_} have been widely studied by many authors. See, for example, [10--23] and the references therein.

Meanwhile, the local and global well-posedness for the INLS equation \eqref{GrindEQ__1_1_} in the fractional Sobolev space $H^{s} (\mathbb R^{n} )$ have also been attracted a lot of interest in recent years. See [1, 24--28] and the references therein. The $H^{s}$-subcritical case was studied by \cite{AK211, G17, AKC21, AK212, AT21}.
Guzm\'{a}n \cite{G17} established the local and global well-posedness in $H^{s} (\mathbb R^{n} )$ with $0\le s\le \min \left\{1,\;\frac{n}{2} \right\}$ for the subcritical INLS equation \eqref{GrindEQ__1_1_} with $f(u)=\lambda |u|^{\sigma } u$. More precisely, he proved that:
\begin{itemize}
  \item if $0<\sigma <\frac{4-2b}{n}$, and $0<b<\min \{ 2,\;n\} $, then \eqref{GrindEQ__1_1_} is globally well-posed in $L^{2} (\mathbb R^{n} )$;
  \item  if $0<s\le \min \left\{1,\;\frac{n}{2} \right\}$, $0<b<\tilde{2}$ and $0<\sigma <\sigma _{s} $, then \eqref{GrindEQ__1_1_} is locally well-posed in $H^{s}(\mathbb R^{n})$;
  \item if $0<s\le \min \left\{1,\;\frac{n}{2} \right\}$, $0<b<\tilde{2}$ and $\frac{4-2b}{n} <\sigma <\sigma _{s} $, then \eqref{GrindEQ__1_1_} is globally well-posed in $H^{s}(\mathbb R^{n})$ for small initial data, where
\end{itemize}
\begin{equation} \nonumber
\tilde{2}=\left\{\begin{array}{l} {\frac{n}{3} ,\;n=1,\;2,\;3,} \\ {2,\;n\ge 4.} \end{array}\right.
\end{equation}
Later, the authors in \cite{AK211} improved the local well-posedness result of \cite{G17} by extending the validity of $s$ and $b$. More precisely, they obtained the following result. See also \cite{AKC21, AK212}.
\begin{theorem}[\cite{AK211}]\label{thm 1.3.}
Let $n\in \mathbb N$, $0\le s<\min \left\{n,\;\frac{n}{2} +{1}\right\}$, $0<b<\min \left\{2,\; n-s,\; 1+\frac{n-2s}{2} \right\}$ and $0<\sigma <\sigma _{s} $. Assume that $f$ is of class $X \left(\sigma,s,b\right)$. Then for any $u_{0} \in H^{s} (\mathbb R^{n})$, there exist $T_{\max } =T_{\max } \left(\left\| u_{0} \right\| _{H^{s} } \right)>0$, $T_{\min } =T_{\min } \left(\left\| u_{0} \right\| _{H^{s} } \right)>0$ and a unique, maximal solution of \eqref{GrindEQ__1_1_} satisfying
\begin{equation} \label{GrindEQ__1_8_}
C\left(\left(-T_{\min } ,\;T_{\max } \right),H^{s} \right)\cap L_{\rm loc}^{\gamma (p)} \left(\left(-T_{\min } ,\;T_{\max } \right),\;H_{p}^{s} (\mathbb R^{n})\right)
\end{equation}
for any admissible pair $(\gamma (p),\;p)$.
If $s>0$, then the solution of \eqref{GrindEQ__1_1_} depends continuously on the initial data $u_{0} $ in the following sense. There exists $0<T<T_{\max } ,\, T_{\min } $ such that if $u_{0}^{m} \to u_{0} $ in $H^{s} (\mathbb R^{n})$ and if $u_{m} $ denotes the solution of \eqref{GrindEQ__1_1_} with the initial data $u_{0}^{m} $, then $0<T<T_{\max } \left(u_{0}^{m} \right),\, T_{\min } \left(u_{0}^{m} \right)$ for all sufficiently large $m$ and $u_{m} \to u$ in $L^{\gamma (p)} \left(\left[-T,\;T\right],\;H_{p}^{s-\varepsilon } (\mathbb R^{n})\right)$ as $m\to \infty $ for all $\varepsilon >0$ and all admissible pair $\left(\gamma (p),\, p\right)$. In particular, $u_{m} \to u$ in $C\left(\left[-T,\;T\right],\; H^{s-\varepsilon } \right)$ for all $\varepsilon >0$.
\end{theorem}
But the above result about the continuous dependence in $H^{s}$ with $s>0$ is weaker than what would be ``standard", i.e. $\varepsilon =0$. The standard continuous dependence in $H^{s}$ with $s>0$ for the subcritical INLS equation \eqref{GrindEQ__1_1_} was studied in \cite{AKC21}.
More precisely they obtained the following result.
\begin{theorem}[\cite{AKC21}]\label{thm 1.4.}
Let $n\in \mathbb N$, $0<s<\min \left\{n,\;1+\frac{n}{2} \right\}$, $0<b<\min \left\{2,\; n-s,\; 1+\frac{n-2s}{2} \right\}$ and $0<\sigma <\sigma _{s} $.
Assume that $f$ is of class $C\left(\sigma ,s,b\right)$ (see Definition \ref{defn 1.5.}). Then for any given $u_{0} \in H^{s} (\mathbb R^{n})$, the corresponding solution $u$ of the INLS equation \eqref{GrindEQ__1_1_} in Theorem \ref{thm 1.3.} depends continuously on the initial data $u_{0} $ in the following sense.
For any interval $\left[-S,\, T\right]\subset \left(-T_{\min } (u_{0}),\;T_{\max } (u_{0})\right)$, and every admissible pair $\left(\gamma (p),\;p\right)$, if $u_{0}^{m} \to u_{0} $ in $H^{s} (\mathbb R^{n})$ and if $u_{m} $ denotes the solution of \eqref{GrindEQ__1_1_} with the initial data $u_{0}^{m} $, then $u_{m} \to u$ in $L^{\gamma (p)} (\left[-S,\;T\right],\;H_{p}^{s} (\mathbb R^{n} ))$ as $m\to \infty $.
In particular, $u_{m} \to u$ in $C\left(\left[-S,\;T\right],\;H^{s} \right)$.
In addition, if $f(u)$ is a polynomial in $u$ and $\bar{u}$, or if $f(u)$ is not a polynomial and $\sigma \ge \left\lceil s\right\rceil $, then the dependence is locally Lipschitz.
\end{theorem}
\begin{definition}[\cite{AKC21}]\label{defn 1.5.}
\textnormal{Let $f:\mathbb C\to \mathbb C$, $s>0$, $\sigma>0$, $0\le b<2$.
We say that $f$ is of class $C\left(\sigma ,s,b\right)$ if it satisfies one of the following conditions:}
\begin{itemize}
\item  \textnormal{$f\left(z\right)$ is a polynomial in $z$ and $\bar{z}$ satisfying $1<\deg (f)=1+\sigma< 1+\sigma_{s}$.}

\item  \textnormal{$\left\lceil s\right\rceil -1< \sigma<\sigma_{s}$ and $f\in C^{\left\lceil s\right\rceil } \left(\mathbb C\to \mathbb C\right)$ satisfies \eqref{GrindEQ__1_3_}
  for any $0\le k\le \left\lceil s\right\rceil $ and $u\in \mathbb C$. Furthermore,
  \begin{equation} \label{GrindEQ__1_9_}
  \left|f^{\left(\left\lceil s\right\rceil\right)} (u)-f^{\left(\left\lceil s\right\rceil\right)} (v)\right|\lesssim|u-v|^{\min \{ \sigma -\left\lceil s\right\rceil +1,\;1\} } \left(|u|+|v|\right)^{\max \{ 0,\;\sigma -\left\lceil s\right\rceil \} } ,
  \end{equation}
  for any $u,\;v\in \mathbb C$. If $s<1$, assume further that $\sigma >1$ and $f\in C^{2} \left(\mathbb C\to \mathbb C\right)$ satisfies
  \begin{equation} \label{GrindEQ__1_10_}
  \left|f''(u)\right|\lesssim|u|^{\sigma -1},
  \end{equation}
  for any $u\in \mathbb C$.}
\end{itemize}
\end{definition}
\begin{remark}\label{rem 1.6.}
\textnormal{Let $s>0$, $0\le b<2$ and $0<\sigma <\sigma _{s} $. If $\sigma $ is not an even integer, assume that $\sigma>\left\lceil s\right\rceil-1 $. If $s<1$, suppose further that $\sigma > 1$. Then one can verify that $f(u)=\lambda |u|^{\sigma } u$ with $\lambda \in \mathbb C$ is a model case of class $C\left(\sigma ,s,b\right)$. See also \cite{CFH11, DYC13} for example.}
\end{remark}
Recently, the authors in \cite{AK212} prove that the subcritical INLS equation \eqref{GrindEQ__1_1_} is globally well-posed in $H^{s}(\mathbb R^{n})$ if $0<s<\min \left\{\frac{n}{2} +1,\; n\right\}$, $0<b<\min \left\{2,\; n-s,\; 1+\frac{n-2s}{2} \right\}$, $\frac{4-2b}{n}<\sigma <\sigma _{s}$ and the initial data is sufficiently small.

The local and global well-posedness in $H^{s}$ for the $H^{s}$-critical INLS equation \eqref{GrindEQ__1_1_} has also been studied by several authors. See, for example, \cite{AK213,AT21} and the references therein.
The authors in \cite{AK213} proved that the critical INLS equation \eqref{GrindEQ__1_1_} with $n\ge 3$, $1\le s<\frac{n}{2}$, $0<b<1+\frac{n-2s}{2}$ and $\sigma =\frac{4-2b}{n-2s}$ is locally well-posed in $H^{s}(\mathbb R^{n})$ if one of the following assumptions is satisfied:
\begin{itemize}
  \item $s\in \mathbb N$, $n\ge 3$ and $b<\frac{4s}{n} $,
  \item $s\notin \mathbb N$, $n\ge 4$ and $b<\frac{6s}{n} -1$,
  \item $s\notin \mathbb N$, $n=3$ and $b<1$.
\end{itemize}
However, the above conditions on $b$ is too restrictive in high dimension. Recently,  Aloui-Tayachi \cite{AT21} used the Strichartz estimates in Sobolev-Lorentz spaces to prove that \eqref{GrindEQ__1_1_} is locally well-posed in $H^{s}$ if $n\in \mathbb N$, $0\le s\le 1$, $s<\frac{n}{2}$, $0<b<\min\{2,\;n-2s\}$ and $0<\sigma\le\frac{4-2b}{n-2s}$. They also proved that the above solution depends continuously on the initial data if one of the following conditions is further satisfied:
\begin{itemize}
  \item $s=0$, $s>0$ and $\sigma>1$, or $\sigma=s=1$,
  \item $s=1$ and $\sigma<\min\{\frac{4-2b}{n-2},\;1\}$,
  \item $s=1$ and $\sigma=\frac{4-2b}{n-2}<1$.
\end{itemize}
Furthermore, they obtained the unconditional uniqueness result in the following cases:
\begin{itemize}
  \item $n\ge 1$, $s\ge \frac{n}{2}$, $0<b<\min\{2,\;n\}$ and $\sigma>0$,
  \item $n\le 2$, $0\le s<\frac{n}{2}$, $0<b<\frac{n+2s}{2}$ and $0<\sigma<\frac{n+2s-2b}{n-2s}$,
  \item $n\ge 3$, $0\le s<\frac{n}{2}$, $0<b<\min\{2,\;1+2s-\frac{2s}{n}\}$ and $0<\sigma<\min\{\frac{2+4s-2b-4s/n}{n-2s},\;\frac{4-2b}{n-2s}\}$,
  \item $n\ge3$, $1\le s<\frac{n}{2}$, $0<b<2$ and $\sigma=\frac{4-2b}{n-2s}$.
\end{itemize}

Inspired by \cite{AT21}, this paper mainly investigates the local well-posedness and standard continuous dependence in $H^{s}$ with $0\le s<\frac{n}{2}$ for the $H^{s}$-critical INLS equation \eqref{GrindEQ__1_1_} by using the contraction mapping principle combined with the Strichartz estimates in Sobolev-Lorentz spaces. Our main results (Theorem \ref{thm 1.7.} and Theorem \ref{thm 1.10.}) improve the ones of \cite{AK213, AT21} by extending the validity of $s$ and $b$.

The first main result of this paper is the following.
\begin{theorem}\label{thm 1.7.}
Let $n\in \mathbb N$, $0\le s<\frac{n}{2} $, $0<b<\min \left\{2,\; n-s,\; 1+\frac{n-2s}{2} \right\}$ and $\sigma
=\frac{4-2b}{n-2s} $. Assume that $f$ is of class $X \left(\sigma,s,b\right)$. If $u_{0} \in H^{s} (\mathbb R^{n})$, then there exist $T_{\max }=T_{\max }(u_{0})>0$ and $T_{\min }=T_{\min }(u_{0})>0$ such that \eqref{GrindEQ__1_1_} has a unique, maximal solution
\begin{equation} \label{GrindEQ__1_11_}
u\in L^{\gamma (r)}_{\rm loc} \left(\left(-T_{\min } ,\;T_{\max } \right),\;H_{r,2}^{s} (\mathbb R^{n})\right),
\end{equation}
where $\left(\gamma (r),\;r\right)$ is an admissible pair satisfying
\footnote[1]{$a^{+}$ is a fixed number slightly larger than $a$ ($a^{+}=a+\varepsilon$ with $\varepsilon>0$ small enough).
}
\begin{equation} \label{GrindEQ__1_12_}
r=\left\{\begin{array}{l} {\left(\frac{\sigma n+n}{\sigma s+n-b}\right)^{+},\;n=1,\;2,} \\ {\frac{2n\sigma +2n}{n+2+2\sigma s -2b},\;n\ge 3.} \end{array}\right.
\end{equation}
Moreover, for any admissible pair $\left(\gamma (p),\;p\right)$, we have
\begin{equation} \label{GrindEQ__1_13_}
u\in C\left(\left(-T_{\min } ,\;T_{\max } \right),H^{s} \right)\cap L^{\gamma (p)}_{\rm loc} \left(\left(-T_{\min } ,\;T_{\max } \right),\;H_{p,2}^{s} (\mathbb R^{n})\right).
\end{equation}
If $\left\| u_{0} \right\| _{\dot{H}^{s}(\mathbb R^{n}) } $ is sufficiently small, then the above solution is  global and scatters.
Furthermore, the solution of \eqref{GrindEQ__1_1_} depends continuously on the initial data $u_{0}$ in the following sense. There exists $0<T<T_{\max } ,\, T_{\min } $ such that if $u_{0}^{m} \to u_{0} $ in $H^{s} (\mathbb R^{n})$ and if $u_{m} $ denotes the solution of \eqref{GrindEQ__1_1_} with the initial data $u_{0}^{m} $, then $0<T<T_{\max } \left(u_{0}^{m} \right),\, T_{\min } \left(u_{0}^{m} \right)$ for all sufficiently large $m$ and $u_{m} \to u$ in $L^{\gamma (p)} \left(\left[-T,\;T\right],\;L^{p,2}(\mathbb R^{n})\right)$ as $m\to \infty $ for any admissible pair $\left(\gamma (p),\, p\right)$. Especially, if $s>0$, then $u_{m} \to u$ in $C\left(\left[-T,\;T\right],\; H^{s-\varepsilon } \right)$ for all $\varepsilon >0$.

\end{theorem}
Theorem \ref{thm 1.7.} applies in particular to the model case $f(u)=\lambda|u|^{\sigma}u$ with $\lambda\in\mathbb C$.
\begin{corollary}\label{cor 1.8.}
Let $n\in \mathbb N$, $0\le s<\frac{n}{2} $, $0<b<\min \left\{2,\; n-s,\; 1+\frac{n-2s}{2} \right\}$ and $\sigma=\frac{4-2b}{n-2s} $. Let $f(u)=\lambda|u|^{\sigma}u$ with $\lambda\in\mathbb C$. If $\sigma$ is not an even integer, assume further that $\sigma>\left\lceil s\right\rceil-1$. Then, for any $u_{0} \in H^{s} (\mathbb R^{n})$, \eqref{GrindEQ__1_1_} has a unique, maximal solution satisfying \eqref{GrindEQ__1_11_}--\eqref{GrindEQ__1_13_}. If $\left\| u_{0} \right\| _{\dot{H}^{s}(\mathbb R^{n}) } $ is sufficiently small, then the above solution is  global and scatters. Furthermore, the solution of \eqref{GrindEQ__1_1_} depends continuously on the initial data $u_{0}$ in the sense of Theorem \ref{thm 1.7.}.
\end{corollary}
\begin{remark}\label{rem 1.9.}
\textnormal{Obviously, Corollary \ref{cor 1.8.} improves the result of \cite{AT21}, where the well-posedness was proved in the case $0\le s\le1$, $s<\frac{n}{2}$ and $0<b<\min \left\{2,\;n-2s\right\}$. To improve the validity of $s$, we establish the estimates of the nonlinearity $f(u)$ in Sobolev-Lorentz spaces of order $0\le s<\frac{n}{2}$ (see Lemma \ref{lem 3.1.}), while Aloui-Tayachi \cite{AT21} only used the well-known fractional chain rule (Lemma \ref{lem 2.11.}) which holds for $s\le1$. The improvement in the range of $b$ comes from the best technical choice of $r$ in Theorem \ref{thm 1.7.}, i.e. the choice of working space where we apply the contraction mapping principle.}
\end{remark}

The second main result of this paper concerns with the continuous dependence of the Cauchy problem for the $H^{s}$-critical INLS equation \eqref{GrindEQ__1_1_} in the standard sense in $H^{s}$, i.e. in the sense that the local solution flow is continuous $H^{s} \to H^{s} $.
\begin{theorem}\label{thm 1.10.}
Let $n\in \mathbb N$, $0<s<\frac{n}{2} $, $0<b<\min \left\{2,\; n-s,\; 1+\frac{n-2s}{2} \right\}$ and $\sigma=\frac{4-2b}{n-2s} $.
Assume that one of the following conditions is satisfied:
\begin{itemize}
  \item $f\left(z\right)$ is a polynomial in $z$ and $\bar{z}$ satisfying $\deg (f)=1+\sigma$.
  \item $0<s<1$, $\sigma > 1$ and $f\in C^{2} \left(\mathbb C\to \mathbb C\right)$ satisfies \eqref{GrindEQ__1_3_} for any $0\le k\le 2$ and $u\in \mathbb C$
  \item $s\ge 1$, $\sigma \ge \left\lceil s\right\rceil $ and $f\in C^{\left\lceil s\right\rceil } \left(\mathbb C\to \mathbb C\right)$ satisfies \eqref{GrindEQ__1_3_}
      for any $0\le k\le \left\lceil s\right\rceil $ and $u\in \mathbb C$. Furthermore,
      \begin{equation} \nonumber
      \left|f^{\left(\left\lceil s\right\rceil\right)} (u)-f^{\left(\left\lceil s\right\rceil \right)} (v)\right|\lesssim \left(|u|+|v|\right)^{\sigma -\left\lceil s\right\rceil}|u-v| ,
      \end{equation}
     for any $u,\;v\in \mathbb C$.
\end{itemize}
Then for any given $u_{0} \in H^{s} (\mathbb R^{n})$, the corresponding solution $u$ of the INLS equation \eqref{GrindEQ__1_1_} in Theorem \ref{thm 1.7.} depends continuously on the initial data $u_{0}$ in the following sense.
For any interval $\left[-S,\, T\right]\subset \left(-T_{\min } (u_{0}),\;T_{\max } (u_{0})\right)$, and every admissible pair $\left(\gamma (p),\;p\right)$, if $u_{0}^{m} \to u_{0} $ in $H^{s} (\mathbb R^{n})$ and if $u_{m} $ denotes the solution of \eqref{GrindEQ__1_1_} with the initial data $u_{0}^{m} $, then $u_{m} \to u$ in $L^{\gamma (p)} (\left[-S,\;T\right],\;H_{p,2}^{s} (\mathbb R^{n} ))$ as $m\to \infty $.
In particular, $u_{m} \to u$ in $C\left(\left[-S,\;T\right],\;H^{s} \right)$.
\end{theorem}
Theorem \ref{thm 1.10.} applies in particular to the model case $f(u)=\lambda|u|^{\sigma}u$ with $\lambda\in\mathbb C$.
\begin{corollary}\label{cor 1.11.}
Let $n\in \mathbb N$, $0< s<\frac{n}{2} $, $0<b<\min \left\{2,\; n-s,\; 1+\frac{n-2s}{2} \right\}$ and $\sigma=\frac{4-2b}{n-2s} $. Let $f(u)=\lambda|u|^{\sigma}u$ with $\lambda\in\mathbb C$. If $\sigma$ is not an even integer, assume either $0<s<1$ and $\sigma>1$, or $s\ge 1$ and $\sigma\ge\left\lceil s\right\rceil$. Then for any $u_{0}\in H^{s}$, the corresponding solution the solution of \eqref{GrindEQ__1_1_} depends continuously on the initial data $u_{0}$ in the sense of Theorem \ref{thm 1.10.}.
\end{corollary}
\begin{remark}\label{rem 1.12.}
\textnormal{
\begin{enumerate}
 \item In fact, the continuous dependence in Theorem \ref{thm 1.10.} and Corollary \ref{cor 1.11.} is locally Lipschitz. See the proof of Theorem \ref{thm 1.10.} in Section 4.
 \item When $0<s<1$, the existence result in Corollary \ref{cor 1.8.} holds for  $\sigma>0$, while Corollary \ref{cor 1.11.} holds for $\sigma>1$. The standard continuous dependence in the case $0<s<1$ and $0<\sigma\le 1$ is still open. See also Remark 3 in \cite{AT21}.
 \item When $1< s< \frac{n}{2}$ and $\sigma$ is not an even integer, Corollary \ref{cor 1.8.} holds for $\sigma>\left\lceil s\right\rceil-1$, while Corollary \ref{cor 1.11.} holds for $\sigma\ge \left\lceil s\right\rceil$. The standard continuous dependence in the case $1< s< \frac{n}{2}$ and $\left\lceil s\right\rceil-1<\sigma< \left\lceil s\right\rceil$ is also an open problem.
\end{enumerate}}
\end{remark}

Based on the local well-posedness in $H^{1}$ (see Corollary \ref{cor 1.8.}), we establish the blow-up criteria for $H^{1}$-solutions to the focusing energy-critical INLS equation
\begin{equation} \label{GrindEQ__1_14_}
iu_{t} +\Delta u+|x|^{-b}
|u|^{\sigma} u=0,~u(0)=u_{0}\in H^{1}(\mathbb R^{n}),
\end{equation}
where $n\ge3$, $0<b<2$ and $\sigma:=\sigma_{1}=\frac{4-2b}{n-2}$.
Let $C_{HS}$ be the sharp constant in the Hardy-Sobolev inequality related to the focusing energy-critical INLS equation \eqref{GrindEQ__1_14_}, namely,
\begin{equation}\label{GrindEQ__1_15_}
C_{HS}:=\inf \left\{\left\|f\right\|_{\dot{H}^{1}}\div\left\| |x|^{-b} |u|^{\sigma_{1} +2} \right\|_{L^{1}}^{\frac{1}{\sigma_{1} +2}};\;f\in \dot{H}^{1}\setminus\left\{0\right\}\right\}.
\end{equation}
We will see in Lemma \ref{lem 5.3.} that the sharp constant $C_{HS}$ is attained by function:
\begin{equation}\label{GrindEQ__1_16_}
W_{b}(x):=\frac{\left[\varepsilon(n-b)(n-2)\right]^{\frac{n-2}{4-2b}}}
{\left(\varepsilon+|x|^{2-b}\right)^{\frac{n-2}{2-b}}},
\end{equation}
for all $\varepsilon>0$.

The last main result of this paper is the following blow-up result for the focusing energy-critical INLS equation \eqref{GrindEQ__1_14_}.
\begin{theorem}\label{thm 1.13.}
Let $n\ge 3$, $0<b<\min \left\{2,\;\frac{n}{2}\right\}$ and $\sigma=\frac{4-2b}{n-2}$.
Let $u_{0}\in H^{1}(\mathbb R^{n})$ and $u$ be the corresponding solution to the focusing energy-critical INLS equation \eqref{GrindEQ__1_14_} defined on the maximal time interval of existence $(-T_{*},\;T^{*})$.
Suppose that either $E(u_{0})<0$, or if $E(u_{0})\ge0$, we assume that $E(u_{0})<E(W_{b})$ and $\left\|u_{0}\right\|_{\dot{H}^{1}}>\left\|W_{b}\right\|_{\dot{H}^{1}}$.
Then the solution either blows up in finite time, or there exists a time sequence $(t_{n})_{n\ge 1}$ satisfying $|t_{n}|\to \infty$ such that $\left\|u(t_{n})\right\|_{\dot{H}^{1}}\to \infty$ as $n\to \infty$.
Moreover, if we assume in addition that
 \begin{itemize}
    \item $u_{0}$ has finite-variance,
    \item or $u_{0}$ is radially symmetric,
    \item or $b\ge 4-n$, and $u_{0}$ is cylindrically symmetric, i.e. $u_{0}\in \Sigma_{n}$, where
     \begin{equation}\label{GrindEQ__1_17_}
     \Sigma_{n}:=\left\{f\in H^{1}:\; f(y,x_{n})=f(|y|,x_{n}),\;x_{n}f\in L^{2}\right\}
     \end{equation}
     with $x=(y,x_{n})$, $y=(x_{1},\cdots,x_{n-1})\in \mathbb R^{n-1}$, and $x_{n}\in \mathbb R$,
 \end{itemize}
then the corresponding solution blows up in finite time, i.e., $T_{*},\;T^{*}<\infty$.
\end{theorem}
\begin{remark}\label{rem 1.14.}
\textnormal{Theorem \ref{thm 1.13.} extends the well-known finite time blow-up results of \cite{DWZ16,KM06,KV10} for the focusing energy-critical NLS equation (i.e. \eqref{GrindEQ__1_14_} with $b = 0$) to the focusing energy-critical INLS equation \eqref{GrindEQ__1_14_}.}
\end{remark}

This paper is organized as follows. In Section 2, we recall some useful facts and prove some auxiliary results related to Sobolev-Lorentz spaces. In Section 3, we prove Theorem \ref{thm 1.7.}. In Section 4, we prove Theorem \ref{thm 1.10.}. Theorem \ref{thm 1.13.} is proved in Section 5.

%%%%%%%%%%%%%%%%%%%%%%%%%%%%%%%%%%%%%%%%%%%%%%%%%%%%%%%%%%%%%%%%%%%%%%%%%%%%%
\section{Preliminaries}
Let us introduce some notation used throughout the paper. As usual, we use $\mathbb C$, $\mathbb R$ and $\mathbb N$ to stand for the sets of complex, real and natural numbers, respectively. $\mathscr{F}$ denotes the Fourier transform, and the inverse Fourier transform is denoted by $\mathscr{F}^{-1}$. We also use the notation $\hat{f}$ instead of $\mathscr{F}f$. $C>0$ stands for a positive universal constant, which can be different at different places. The notation $a\lesssim b$ means $a\le Cb$ for some constant $C>0$. For $p\in \left[1,\;\infty \right]$, $p'$ denotes the dual number of $p$, i.e. $1/p+1/p'=1$.
For $s\in \mathbb R$, we denote by $\left[s\right]$ the largest integer which is less than or equal to $s$ and by $\left\lceil s\right\rceil$ the minimal integer which is larger than or equal to $s$. For a multi-index $\alpha =\left(\alpha _{1} ,\;\alpha _{2},\;\ldots ,\;\alpha _{n} \right)$, denote
$$
D^{\alpha } =\partial _{x_{1} }^{\alpha _{1} } \cdots \partial _{x_{n} }^{\alpha _{n} } , \;\left|\alpha \right|=\left|\alpha _{1} \right|+\cdots +\;\left|\alpha _{n} \right|, \;\xi^{\alpha } =\xi_{1}^{\alpha _{1} } \cdots \xi_{n}^{\alpha _{n} } .
$$
For a function $f(z)$ defined for a complex variable $z$ and for a positive integer $k$, the $k$-th order derivative of $f(z)$ and its norm are defined by
$$
f^{(k)}(z):=\left(\frac{\partial ^{k} f}{\partial z^{k} } ,\; \frac{\partial ^{k} f}{\partial z^{k-1} \partial \bar{z}} ,\; {\dots},\;\frac{\partial ^{k} f}{\partial \bar{z}^{k} } \right),~\left|f^{\left(k\right)} (z)\right|:=\sum _{i=0}^{k}\left|\frac{\partial ^{k} f}{\partial z^{k-i} \partial \bar{z}^{i} } \right|,
$$
where
$$
\frac{\partial f}{\partial z}=\frac{1}{2} \left(\frac{\partial f}{\partial x} -i\frac{\partial f}{\partial y} \right),\; \frac{\partial f}{\partial \bar{z}}=\frac{1}{2} \left(\frac{\partial f}{\partial x} +i\frac{\partial f}{\partial y} \right).
$$
As in \cite{WHHG11}, for $s\in \mathbb R$ and $1<p<\infty $, we denote by $H_{p}^{s} (\mathbb R^{n} )$ and $\dot{H}_{p}^{s} (\mathbb R^{n} )$ the nonhomogeneous Sobolev space and homogeneous Sobolev space, respectively. The norms of these spaces are given as
$$
\left\| f\right\| _{H_{p}^{s} (\mathbb R^{n} )} =\left\| J^{s} f\right\| _{L^{p} (\mathbb R^{n} )} , \;\left\| f\right\| _{\dot{H}_{p}^{s} (\mathbb R^{n} )} =\left\| I^{s} f\right\| _{L^{p} (\mathbb R^{n} )},
$$
where $J^{s} =\mathscr{F}^{-1} \left(1+|\xi|^{2} \right)^{\frac{s}{2} } \mathscr{F}$ and $I^{s} =\mathscr{F}^{-1} |\xi|^{s} \mathscr{F}$. As usual, we abbreviate $H_{2}^{s} (\mathbb R^{n} )$ and $\dot{H}_{2}^{s} (\mathbb R^{n} )$ as $H^{s} (\mathbb R^{n} )$ and $\dot{H}^{s} (\mathbb R^{n} )$, respectively. For
$0<p,\; q\le \infty $, we denote by $L^{p,q} \left(\mathbb R^{n}
\right)$ the Lorentz space.
The quasi-norms of these spaces are given by
$$
\left\|f\right\|_{L^{p,q} (\mathbb R^{n})}=:\left(\int_{0}^{\infty}{\left(t^{\frac{1}{p}}f^{*}(t)\right)^{q}
\frac{dt}{t}}\right)^{\frac{1}{q}},~~\textnormal{when}~~0<q<\infty,
$$
$$
\left\|f\right\|_{L^{p,\infty} (\mathbb R^{n})}:=\sup_{t>0}t^{\frac{1}{p}}f^{*}(t),~~\textnormal{when}~~q=\infty,
$$
where $f^{*}(t)=\inf\left\{\tau:M^{n}\left(\left\{x:|f(x)|>\tau\right\}\right)\le t\right\}$, with $M^{n}$ being the Lebesgue measure in $\mathbb R^{n}$. Note that $L^{p,q} \left(\mathbb R^{n}
\right)$ is a quasi-Banach space for $0<p,\; q\le \infty $.  When $1<p<\infty$ and $1\le q \le \infty$,  $L^{p,q} \left(\mathbb R^{n}
\right)$ can be turned into a Banach space via an equivalent norm.  Note also that $L^{p,p} (\mathbb R^{n})=L^{p}
(\mathbb R^{n})$. See \cite{G14} for details. For $I\subset \mathbb R$ and $\gamma \in \left[1,\;\infty \right]$, we will use the space-time mixed space $L^{\gamma } \left(I,X\left(\mathbb R^{n}
\right)\right)$ whose norm is defined by
\[\left\| f\right\|_{L^{\gamma } \left(I,\;X(\mathbb R^{n})\right)}
=\left(\int _{I}\left\| f\right\| _{X(\mathbb R^{n})}^{\gamma } dt
\right)^{\frac{1}{\gamma } } ,\]
with a usual modification when $\gamma =\infty $, where $X(\mathbb R^{n})$
is a normed space on $\mathbb R^{n} $. Given normed spaces $X$ and $Y$, $X\subset Y$ means that $X$ is continuously embedded in $Y$, i.e. there exists a constant $C\left(>0\right)$ such that $\left\| f\right\| _{Y} \le C\left\| f\right\| _{X} $ for all $f\in X$. If there is no confusion, $\mathbb R^{n} $ will be omitted in various function spaces.

We recall some useful facts about Lorentz spaces.

\begin{lemma}[\cite{G14}]\label{lem 2.1.}
For $0<p<\infty$, $|x|^{-\frac{n}{p} } $ is in
$L^{p,\infty } (\mathbb R^{n})$ with the norm $v_{n}^{1/p}$, where $v_{n}$ is the measure of the unit ball of $\mathbb R^{n}$.
\end{lemma}

\begin{lemma}[\cite{G14}]\label{lem 2.2.}
For all $0<p,\;r<\infty $, $0<q\le
\infty $ we have
\[
\left\| \left|f\right|^{r} \right\| _{L^{p,q} } =\left\| f\right\| _{L^{pr,qr}
}^{r}.\]
\end{lemma}
\begin{lemma}[\cite{G14}]\label{lem 2.3.}
Suppose $0<p\le \infty $ and $0<q<r\le \infty $. Then we have
\[\left\| f\right\| _{L^{p,r} } \le C_{p,q,r} \left\| f\right\| _{L^{p,q} } .\]
\end{lemma}
Next, we recall the H\"{o}lder inequality and Young inequality for Lorentz spaces. See \cite{O63} for example.

\begin{lemma}[H\"{o}lder inequality for Lorentz spaces]\label{lem 2.4.}
Let $1<p,\;q,\;r<\infty $, $1\le s,\;s_{1} ,\;s_{2}\le \infty $. Assume that
\[\frac{1}{p} +\frac{1}{q} =\frac{1}{r} ,~\frac{1}{s_{1} } +\frac{1}{s_{2} }
=\frac{1}{s} .\]
Then we have
\[\left\| fg\right\| _{L^{r,s} } \lesssim \left\| f\right\|
_{L^{p,s_{1} } } \left\| g\right\| _{L^{q,s_{2} } } .\]
\end{lemma}
\begin{lemma}[Young inequality for Lorentz spaces]\label{lem 2.5.}
Let $1<p,\;q,\;r<\infty $, $1\le s,\;s_{1} ,\;s_{2} \le \infty $. Assume that
\[\frac{1}{p} +\frac{1}{q} =\frac{1}{r} +1,~\frac{1}{s_{1} } +\frac{1}{s_{2} } =\frac{1}{s} .\]
Then we have
\[\left\| f*g\right\| _{L^{r,s} } \lesssim \left\| f\right\| _{L^{p,s_{1} } } \left\| g\right\| _{L^{q,s_{2} } } .\]
\end{lemma}
In this paper, we mainly use the homogeneous and nonhomogeneous Sobolev-Lorentz spaces which are defined as follows. See \cite{AT21, HYZ12} for example.
\begin{definition}\label{defn_2.6.}
\textnormal{Let $s\ge 0$, $1<p<\infty$ and $1\le q \le \infty$. The homogeneous Sobolev-Lorentz space $\dot{H}^{s}_{p,q}(\mathbb R^{n})$ is defined as the set of functions satisfying $I^{s}f \in L^{p,q}(\mathbb R^{n})$, equipped with the norm
\[\left\|f\right\|_{\dot{H}^{s}_{p,q}(\mathbb R^{n})}:=\left\|I^{s}f\right\|_{L^{p,q}(\mathbb R^{n})}.\]
The nonhomogeneous Sobolev-Lorentz space ${H}^{s}_{p,q}(\mathbb R^{n})$ is defined as the set of functions satisfying $J^{s}f \in L^{p,q}(\mathbb R^{n})$, equipped with the norm
\[\left\|f\right\|_{{H}^{s}_{p,q}(\mathbb R^{n})}:=\left\|f\right\|_{L^{p,q}(\mathbb R^{n})}+\left\|I^{s}f\right\|_{L^{p,q}(\mathbb R^{n})}.\]}
\end{definition}
\begin{lemma}\label{lem 2.7.}
Let $s\ge 0$, $1<p<\infty$ and $1\le q_{1}\le q_{2}\le \infty$. Then we have

$(a)$ $\dot{H}^{s}_{p,1}\subset \dot{H}^{s}_{p,q_{1}} \subset \dot{H}^{s}_{p,q_{2}} \subset \dot{H}^{s}_{p,\infty}$.

$(b)$ $\dot{H}^{s}_{p,p}=\dot{H}^{s}_{p}$.
\end{lemma}
\begin{proof}
It is easily deduced from the properties of Lorentz spaces.
\end{proof}
\begin{lemma}\label{lem 2.8.}
Let $0\le s_{2} \le s_{1} <\infty $ and $1<p_{1} \le p_{2} <\infty $ with $s_{1} -\frac{n}{p_{1} } =s_{2} -\frac{n}{p_{2} } $. Then, for any $1\le q\le \infty$, there holds the embedding: $\dot{H}_{p_{1},q}^{s_{1} } \subset \dot{H}_{p_{2},q}^{s_{2} }$.
\end{lemma}
\begin{proof}
It follows from Lemma \ref{lem 2.1.} and Lemma \ref{lem 2.5.} that
\begin{eqnarray}\begin{split}\nonumber
\left\|f\right\|_{\dot{H}_{p_{2},q}^{s_{2}} }&=\left\|\mathscr{F}^{-1}\left(|\xi|^{-(s_{1}-s_{2})}\right)*I^{s_{1}}f\right\|_{L^{p_{2},q}}
=C \left\||x|^{-(n+s_{2}-s_{1})}*I^{s_{1}}f\right\|_{L^{p_{2},q}}\\
& \lesssim\left\||x|^{-(n+s_{2}-s_{1})}\right\|_{L^{\bar{p},\infty}}
\left\|I^{s_{1}}f\right\|_{L^{p_{1},q}}
\lesssim\left\|f\right\|_{\dot{H}_{p_{1},q}^{s_{1}}},
\end{split}\end{eqnarray}
where $\frac{1}{\bar{p}}=1-\frac{1}{p_{1}}+\frac{1}{p_{2}}=\frac{n+s_{2}-s_{1}}{n}$.
\end{proof}
\begin{lemma}\label{lem 2.9.}
Let $s\ge 0$, $1<p<\infty $, $1\le q\le \infty$ and $v=s-[s]$. Then we have
\[ \left\|f\right\|_{\dot{H}_{p,q}^{s} }\sim \sum _{\left|\alpha \right|=[s]}\left\|D^{\alpha }f\right\| _{\dot{H}_{p,q}^{v}}.\]
\end{lemma}

\begin{proof}
The case $s<1$ is trivial and we assume that $s\ge 1$. We know that
\begin{equation} \nonumber
D^{\alpha } f=C_{\alpha } \mathscr{F}^{-1} \left(\xi ^{\alpha } \hat{f}(\xi)\right)=C_{\alpha } \mathscr{F}^{-1} \left(\rho_{1} (\xi)|\xi|^{[s]} \hat{f}(\xi)\right),
\end{equation}
where $\rho_{1} (\xi)=\frac{\xi ^{\alpha } }{|\xi|^{[s]} } $. Noticing that $\rho_{1} (\xi)$ is a multiplier on $L^{p} (\mathbb R^{n})$ for any $1<p<\infty$, and using the general Marcinkiewicz interpolation theorem (see Theorem 5.3.2 in \cite{BL76}), we can see that $\rho_{1} (\xi)$ is a multiplier on $L^{p,q} (\mathbb R^{n})$ for any $1<p<\infty$ and $1\le q \le \infty$. Thus we have
\begin{equation}\nonumber
\left\|D^{\alpha }f\right\| _{\dot{H}_{p,q}^{v}}=\left\| C_{\alpha } \mathscr{F}^{-1}\left(|\xi|^{v} \rho_{1} (\xi)|\xi|^{[s]} \hat{f}(\xi)\right)\right\| _{L^{p,q}} =\left\| C_{\alpha } \mathscr{F}^{-1}\left(\rho_{1} (\xi)|\xi|^{s} \hat{f}(\xi)\right)\right\| _{L^{p,q}}
\lesssim \left\|f\right\|_{\dot{H}_{p,q}^{s}}.
\end{equation}
Conversely, we know that
\begin{equation} \nonumber
|\xi|^{[s]} =\frac{\left(\xi _{1}^{2} +\cdots +\xi _{n}^{2} \right)^{[s]} }{|\xi|^{[s]} } =\sum _{\left|\alpha \right|=[s]}C_{\alpha } \xi ^{\alpha } \rho_{1} (\xi).
\end{equation}
 Since $\rho_{1} (\xi)$ is a multiplier on $L^{p,q} (\mathbb R^{n})$, we have
\begin{equation} \nonumber
\mathscr{F}^{-1} \left(|\xi|^{s} \hat{f}\right)=\mathscr{F}^{-1} \left(|\xi|^{v} |\xi|^{[s]} \hat{f}\right)=\sum _{\left|\alpha \right|=[s]}C_{\alpha } \mathscr{F}^{-1}\left(\rho_{1} (\xi)|\xi|^{v} \xi ^{\alpha } \hat{f}\right) ,
\end{equation}
which implies that
\begin{equation} \nonumber
\left\|f\right\|_{\dot{H}_{p,q}^{s}}\lesssim \sum _{\left|\alpha \right|=[s]}\left\| \mathscr{F}^{-1}\left(\rho_{1} (\xi)|\xi|^{v} \xi ^{\alpha } \hat{f}\right)\right\| _{L^{p,q}}  \lesssim \sum _{\left|\alpha \right|=[s]}\left\|D^{\alpha } f\right\| _{\dot{H}_{p,q}^{v}}.
\end{equation}
\end{proof}
\begin{lemma}[Fractional Product Rule in Lorentz spaces, \cite{CN16}]\label{lem 2.10.}
Let $0\le s\le1$, $1<p,\;p_{1},\;p_{2}, \;p_{3},\;\;p_{4} <\infty$ and $1\le q,\;q_{1},\;q_{2}, \;q_{3},\;\;q_{4} \le \infty$. Assume that
$$
\frac{1}{p} =\frac{1}{p_{1} }+\frac{1}{p_{2}}=\frac{1}{p_{3} }+\frac{1}{p_{4}},~\frac{1}{q} =\frac{1}{q_{1} }+\frac{1}{q_{2}}=\frac{1}{q_{3} }+\frac{1}{q_{4}}.
$$
Then we have
\begin{equation} \label{GrindEQ__2_1_}
\left\| fg\right\| _{\dot{H}_{p,q}^{s} } \lesssim \left\| f\right\| _{L^{p_{1},q_{1}}} \left\| g\right\| _{\dot{H}_{p_{2},q_{2}}^{s} } +\left\| f\right\| _{\dot{H}_{p_{3},q_{3}}^{s} } \left\| g\right\| _{L^{p_{4},q_{4}}} .
\end{equation}
\end{lemma}
\begin{lemma}[Fractional Chain Rule in Lorentz spaces, \cite{AT21}]\label{lem 2.11.}
Suppose $G\in C^{1} (\mathbb C, \mathbb C)$ and $0\le s\le1$. Then for $1<p,\;p_{1},\;p_{2} <\infty $, and $1\le q,\;q_{1},\; q_{2}< \infty $ satisfying
$$
\frac{1}{p} =\frac{1}{p_{1} } +\frac{1}{p_{2} },~\frac{1}{q} =\frac{1}{q_{1} } +\frac{1}{q_{2} },
$$
we have
\begin{equation} \label{GrindEQ__2_2_}
\left\| G(u)\right\| _{\dot{H}_{p,q}^{s} } \lesssim \left\| G'(u)\right\| _{L^{p_{1},q_{1}}} \left\|u\right\|_{\dot{H}_{p_{2},q_{2} }^{s} } .
\end{equation}
\end{lemma}

We end this section with recalling the Strichartz estimates in Sobolev-Lorentz spaces. See \cite{AT21, KT98} for instance.
\begin{definition}\label{defn 2.12.}
\textnormal{A pair $(\gamma(p),p)$ is said to be Schr\"{o}dinger admissible if
\begin{equation} \label{GrindEQ__2_3_}
\left\{\begin{array}{l}{2\le p\le\frac{2n}{n-2},\;n\ge 3,}\\{2\le p<\infty,\;n=2,}\\{2\le p\le\infty,\;n=1,} \end{array}\right.
\end{equation}
and
\begin{equation} \label{GrindEQ__2_4_}
\frac{2}{\gamma (p)} =\frac{n}{2} -\frac{n}{p} .
\end{equation}}
\end{definition}
\begin{lemma}[Strichartz estimates in Sobolev-Lorentz spaces]\label{lem 2.13.}
Let $S(t)=e^{it\Delta } $ and $s\ge 0$. Then for any admissible pairs $(\gamma(p),p)$ and $(\gamma(r),r)$, we have
\begin{equation} \label{GrindEQ__2_5_}
\left\| S(t)\phi \right\| _{L^{\gamma(p)}(\mathbb R,\;\dot{H}_{p,2}^{s})} \lesssim\left\| \phi \right\| _{\dot{H}^{s} } ,
\end{equation}
\begin{equation} \label{GrindEQ__2_6_}
\left\|\int_{0}^{t}S(t-\tau)f(\tau)d\tau\right\|_{L^{\gamma(p)}(\mathbb R,\;\dot{H}_{p,2}^{s})}\lesssim\left\|f\right\|_{L^{\gamma(r)'}(\mathbb R,\;\dot{H}_{r',2}^{s})}.
\end{equation}
\end{lemma}
%%%%%%%%%%%%%%%%%%%%%%%%%%%%%%%%%%%%%%%%%%%%%%%%%%%%%%%%%%%%%%%%%%%%%%%%%%%%%
\section{Well-posedness}

In this section, we prove Theorem \ref{thm 1.7.}.
First of all, we estimates the nonlinearity $f(u)$ that behaves like $\lambda |u|^{\sigma } u$ with $\lambda \in \mathbb C$ in the fractional Sobolev-Lorentz spaces.
\begin{lemma}\label{lem 3.1.}
Let $1<p,\;r<\infty $, $s\ge 0$, and $\sigma>\max\{0,\;\left\lceil s\right\rceil-1\}$. Assume that $f\in C^{\left\lceil s\right\rceil } $ satisfies following condition:
\begin{equation} \label{GrindEQ__3_1_}
\left|f^{\left(k\right)} \left(z\right)\right|\lesssim \left|z\right|^{\sigma +1-k},
\end{equation}
for any $0\le k\le\left\lceil s\right\rceil $ and $z\in \mathbb C$. Suppose also that
\begin{equation} \label{GrindEQ__3_2_}
\frac{1}{p} =\sigma \left(\frac{1}{r} -\frac{s}{n} \right)+\frac{1}{r} , \;\frac{1}{r} -\frac{s}{n} >0.
\end{equation}
Then we have
\begin{equation} \label{GrindEQ__3_3_}
\left\| f(u)\right\| _{\dot{H}_{p,2}^{s} } \lesssim \left\| u\right\| _{\dot{H}_{r,2}^{s} }^{\sigma +1} .
\end{equation}
\end{lemma}

\begin{proof}
First, we consider the case $s\le 1$. It follows from \eqref{GrindEQ__3_1_}, Lemmas \ref{lem 2.2.}--\ref{lem 2.4.}, \ref{lem 2.8.} and \ref{lem 2.11.} that
\begin{equation}\nonumber
\left\| f(u)\right\| _{\dot{H}_{p,2}^{s} }\lesssim \left\| u\right\| _{L^{a,q_{1}}}^{\sigma} \left\|u\right\|_{\dot{H}_{r,q_{2} }^{s} }\lesssim \left\| u\right\| _{L^{a,2}}^{\sigma} \left\|u\right\|_{\dot{H}_{r,2 }^{s} }\lesssim  \left\| u\right\| _{\dot{H}_{r,2 }^{s} }^{\sigma+1},
\end{equation}
where
\begin{equation}\label{GrindEQ__3_4_}
\frac{1}{a}=\frac{1}{r}-\frac{s}{n},~
\frac{1}{2}=\frac{\sigma}{q_{1}}+\frac{1}{q_{2}},~ 2<q_{1},\;q_{2}<\infty.
\end{equation}
Next, we consider the case $s>1$. Lemma \ref{lem 2.9.} yields that
\[\left\| f(u)\right\| _{\dot{H}_{p,2}^{s} } \lesssim \sum _{\left|\alpha \right|=[s]}\left\| D^{\alpha } f(u)\right\| _{\dot{H}_{p,2}^{v} }  ,\]
where $v=s-[s]$. Without loss of generality and for simplicity, we assume that $f$ is a function of a real variable. It follows from the Leibniz rule of derivatives that
\begin{equation} \label{GrindEQ__3_5_}
D^{\alpha } f(u)=\sum _{q=1}^{\left|\alpha \right|}\sum _{\Lambda _{\alpha }^{q} }C_{\alpha ,\;q} f^{(q)} (u)\prod _{i=1}^{q}D^{\alpha _{i} } u,
\end{equation}
where $\Lambda _{\alpha }^{q} =\left(\alpha _{1} +\cdots +\alpha _{q} =\alpha ,\;\left|\alpha _{i} \right|\ge 1\right)$. Hence it suffices to show that
\begin{equation} \label{GrindEQ__3_6_}
\left\| f^{(q)} (u)\prod _{i=1}^{q}D^{\alpha _{i} } u \right\| _{\dot{H}_{p,2}^{v} } \lesssim \left\| u\right\| _{\dot{H}_{r,2}^{s} }^{\sigma +1} ,
\end{equation}
where $\left|\alpha _{1} \right|+\cdots +\left|\alpha _{q} \right|=[s]$, $\left|\alpha _{i} \right|\ge 1$, $[s]\ge q\ge 1$ and $v=s-[s]$.
In view of \eqref{GrindEQ__3_2_}, we have
\begin{equation} \label{GrindEQ__3_7_}
\frac{1}{p} =\sigma \left(\frac{1}{r} -\frac{s}{n} \right)+\frac{1}{r} =\frac{\sigma +1-q}{a}+\sum _{i=1}^{q}\frac{1}{a_{i}} +\frac{v}{n},
\end{equation}
where $a$ is given in \eqref{GrindEQ__3_4_} and
\begin{equation} \label{GrindEQ__3_8_}
\frac{1}{a_{i} }=\frac{1}{r} -\frac{s-\left|\alpha _{i} \right|}{n}.
\end{equation}

If $s\in \mathbb N$, then it follows from \eqref{GrindEQ__3_1_}, \eqref{GrindEQ__3_7_},  Lemmas \ref{lem 2.2.}--\ref{lem 2.4.}, \ref{lem 2.7.} and \ref{lem 2.9.} that
\footnote[2]{We assume that $\prod _{i=2}^{1}a_{i}=1$.}
\begin{eqnarray}\begin{split}\nonumber
\left\| f^{(q)} (u)\prod _{i=1}^{q}D^{\alpha _{i} } u \right\| _{L^{p,2}} &\lesssim \left\| u\right\| _{L^{a,q_{1}}}^{\sigma +1-q}\left\| D^{\alpha _{1} } u\right\| _{L^{a_{1},{q_{2}}} } \prod _{i=2}^{q}\left\| D^{\alpha _{i} } u\right\| _{L^{a_{i},q_{1}} } \\
&\lesssim \left\| u\right\| _{L^{a,2}}^{\sigma +1-q} \prod _{i=1}^{q}\left\| u\right\| _{\dot{H}_{a_{i},2}^{\left|\alpha _{i} \right|} }
\lesssim \left\| u\right\| _{\dot{H}_{r,2}^{s} }^{\sigma +1},
\end{split}\end{eqnarray}
where $a$, $a_{i}$, $q_{1}$, $q_{2}$ are given in \eqref{GrindEQ__3_4_} and \eqref{GrindEQ__3_8_}.

Next, we consider the case $s\notin \mathbb N$. In this case, we have $q+1\le \left\lceil s\right\rceil < \sigma+1$. Lemma \ref{lem 2.10.} (fractional product rule) yields that
\begin{eqnarray}\begin{split}\nonumber
\left\| f^{(q)} (u)\prod _{i=1}^{q}D^{\alpha _{i} } u \right\| _{\dot{H}_{p,2}^{v} } &\lesssim \left\| f^{(q)} (u)\right\| _{\dot{H}_{p_{1},q_{3} }^{v} } \left\| \prod _{i=1}^{q}D^{\alpha _{i} } u \right\| _{L^{r_{1},q_{4}}} +\left\| f^{(q)} (u)\right\| _{L^{p_{2},q_{3}} } \left\| \prod _{i=1}^{q}D^{\alpha _{i} } u \right\| _{\dot{H}_{r_{2},q_{4}}^{v} }  \\
&\equiv A_{1} +A_{2} ,
\end{split}\end{eqnarray}
where
\begin{equation} \label{GrindEQ__3_9_}
\frac{1}{p_{1} } =\frac{\sigma +1-q}{a} +\frac{v}{n} , \;\frac{1}{r_{1} } =\sum _{i=1}^{q}\frac{1}{a_{i} }  ,
\end{equation}
\begin{equation} \label{GrindEQ__3_10_}
\frac{1}{p_{2} } =\frac{\sigma +1-q}{a} , \;\frac{1}{r_{2} } =\sum _{i=1}^{q}\frac{1}{a_{i} }  +\frac{v}{n},
\end{equation}
\begin{equation} \label{GrindEQ__3_11_}
\frac{1}{q_{3}}=\frac{\sigma+1-q}{q_{1}},~\frac{1}{q_{4}}=\frac{1}{2}-\frac{1}{q_{3}}
=\frac{1}{q_{2}}+\frac{q-1}{q_{1}}.
\end{equation}
First, we estimate $A_{1}$.
It follows from \eqref{GrindEQ__3_9_}, \eqref{GrindEQ__3_11_}, Lemmas \ref{lem 2.3.}, \ref{lem 2.4.}, \ref{lem 2.8.} and \ref{lem 2.9.} that
\begin{equation} \label{GrindEQ__3_12_}
\left\| \prod _{i=1}^{q}D^{\alpha _{i} } u \right\| _{L^{r_{1},q_{4} }} \le \left\| D^{\alpha _{1} } u\right\| _{a_{1},q_{2}}\prod _{i=2}^{q}\left\| D^{\alpha _{i} } u\right\| _{a_{i},q_{1} }  \le \prod _{i=1}^{q}\left\| u\right\| _{\dot{H}_{a_{i},2 }^{\left|\alpha _{i} \right|} }  \lesssim \left\| u\right\| _{\dot{H}_{r,2}^{s} }^{q} .
\end{equation}
 Since $q<\sigma$, it also follows from Lemma \ref{lem 2.11.} (fractional chain rule) and \eqref{GrindEQ__3_9_} that
\begin{equation} \label{GrindEQ__3_13_}
\left\| f^{(q)} (u)\right\| _{\dot{H}_{p_{1},q_{3} }^{v} } \lesssim \left\| f^{\left(q+1\right)} (u)\right\| _{L^{p_{3},q_{5}} } \left\| u\right\| _{\dot{H}_{p_{4},q_{1}}^{v} } ,
\end{equation}
where $\frac{1}{p_{3} } =\frac{\sigma -q}{a} $, $\frac{1}{p_{4} } =\frac{1}{r} -\frac{[s]}{n} $ and $\frac{1}{q_{5}}=\frac{\sigma-q}{q_{1}}$.
Noticing that $\frac{1}{p_{4}} -\frac{v}{n} =\frac{1}{r} -\frac{s}{n}$, Lemmas \ref{lem 2.3.} and \ref{lem 2.8.} yield that $\dot{H}_{r,2}^{s} \subset \dot{H}_{p_{4},q_{1} }^{v} $. Hence, we immediately get
\begin{equation} \label{GrindEQ__3_14_}
\left\| f^{(q)} (u)\right\| _{\dot{H}_{p_{1},q_{3} }^{v} } \lesssim \left\| u\right\| _{L^{a,q_{1}}}^{\sigma -q} \left\| u\right\| _{\dot{H}_{p_{4},q_{1} }^{v} } \lesssim \left\| u\right\| _{\dot{H}_{r,2}^{s} }^{\sigma -q+1}.
\end{equation}
\eqref{GrindEQ__3_12_} and \eqref{GrindEQ__3_14_} yield that
\begin{equation} \label{GrindEQ__3_15_}
A_{1} =\left\| f^{(q)} (u)\right\| _{\dot{H}_{p_{1},q_{3} }^{v} } \left\| \prod _{i=1}^{q}D^{\alpha _{i} } u \right\| _{L^{r_{1},{q_{4}} }} \lesssim \left\| u\right\| _{\dot{H}_{r,2}^{s} }^{\sigma +1} .
\end{equation}
Next, we estimate $A_{2}$.
It follows from \eqref{GrindEQ__3_1_}, Lemma \ref{lem 2.2.} and the embedding $\dot{H}_{r,2}^{s} \subset L^{a,q_{1}} $ that
\begin{equation} \label{GrindEQ__3_16_}
\left\| f^{(q)} (u)\right\| _{L^{p_{2},q_{3}}}\lesssim \left\| u\right\| _{L^{a,q_{1}}}^{\sigma +1-q} \lesssim \left\| u\right\| _{\dot{H}_{r,2}^{s} }^{\sigma +1-q} .
\end{equation}
If $q=1$, then we can see that $r_{2} =r$, $\left|\alpha _{1} \right|=[s]$ and $q_{4}=q_{2}>2$. Hence it follows from Lemmas \ref{lem 2.7.} and \ref{lem 2.9.} that
\begin{equation} \nonumber
\left\| D^{\alpha _{1} } u\right\| _{\dot{H}_{r_{2},q_{4} }^{v} } \lesssim \left\| u\right\| _{\dot{H}_{r,2}^{s} } .
\end{equation}
We consider the case $q>1$. For $1\le k\le q$, putting $\frac{1}{\tilde{a}_{k} } :=\frac{1}{a_{k} } +\frac{v}{n} $, it follows from \eqref{GrindEQ__3_8_} and \eqref{GrindEQ__3_10_} that
\begin{equation} \label{GrindEQ__3_17_}
\frac{1}{\tilde{a}_{k} } =\frac{1}{r} -\frac{s-\left|\alpha _{k} \right|-v}{n} , \;\frac{1}{r_{2} } =\sum _{i\in I_{k} }\frac{1}{a_{i} }  +\frac{1}{\tilde{a}_{k} } ,
\end{equation}
where $I_{k} =\left\{i\in \mathbb N;\;1\le i\le q,\;i\ne k\right\}$.
\noindent We can see that $\tilde{a}_{k} >r>1$ and $\dot{H}_{r,2}^{s} \subset \dot{H}_{\tilde{a}_{k},q_{2} }^{\left|\alpha _{k} \right|+v} $, since $s>\left|\alpha _{k} \right|+v$.
Thus using \eqref{GrindEQ__3_11_}, \eqref{GrindEQ__3_17_} and Lemma \ref{lem 2.10.}, we have
\begin{equation}\nonumber
\left\| \prod _{i=1}^{q}D^{\alpha _{i} } u \right\| _{\dot{H}_{r_{2},q_{4}}^{v} } \lesssim \sum _{k=1}^{q}\left(\left\| D^{\alpha _{k} } u\right\| _{\dot{H}_{\tilde{a}_{k},q_{2} }^{v} } \prod _{i\in I_{k} }\left\| D^{\alpha _{i} } u_{i} \right\| _{L^{a_{i},q_{1}} }  \right)\lesssim \left\| u\right\| _{\dot{H}_{r,2}^{s} }^{q} ,
\end{equation}
where the last inequality follows from $\dot{H}_{r,2}^{s} \subset \dot{H}_{a_{i},q_{1} }^{\left|\alpha _{i}\right|} $ and $\dot{H}_{r,2}^{s} \subset \dot{H}_{\tilde{a}_{k},q_{2} }^{\left|\alpha _{k} \right|+v} $. Hence, for any $1\le q\le [s]$, we have
\begin{equation} \label{GrindEQ__3_18_}
\left\| \prod _{i=1}^{q}D^{\alpha _{i} } u \right\| _{\dot{H}_{r_{2},q_{4}}^{v} }
\lesssim \left\| u\right\| _{\dot{H}_{r,2}^{s} }^{q} ,
\end{equation}
\eqref{GrindEQ__3_16_} and \eqref{GrindEQ__3_18_} imply that
\begin{equation} \label{GrindEQ__3_19_}
A_{2} =\left\| f^{(q)} (u)\right\| _{L^{p_{2},q_{3}} } \left\| \prod _{i=1}^{q}D^{\alpha _{i} } u \right\| _{\dot{H}_{r_{2},q_{4}}^{v}}\lesssim \left\| u\right\| _{\dot{H}_{r,2}^{s} }^{\sigma +1} .
\end{equation}
In view of \eqref{GrindEQ__3_15_} and \eqref{GrindEQ__3_19_}, we have
\[\left\| f^{(q)} (u)\prod _{i=1}^{q}D^{\alpha _{i} } u \right\| _{\dot{H}_{p,2}^{v} } \lesssim A_{1} +A_{2} \lesssim \left\| u\right\| _{\dot{H}_{r,2}^{s} }^{\sigma +1} ,\]
this completes the proof.
\end{proof}
\begin{remark}[\cite{G17}]\label{rem 3.2.}
\textnormal{Let $b>0$, $s\ge 0$ and $b+s<n$. Then we have $I^{s}(|x|^{-b})=C_{n,b}|x|^{-b-s}$.}
\end{remark}
Using Lemma \ref{lem 3.1.} and Remark \ref{rem 3.2.}, we establish the estimates of the nonlinearity $|x|^{-b}f(u)$ in the fractional Sobolev-Lorentz spaces.
\begin{lemma}\label{lem 3.3.}
Let $1<p,\;r<\infty $, $b>0$, $s\ge 0$, $b+s<n$ and $\sigma>\max \left\{0,\;\left\lceil
s\right\rceil -1\right\}$.
Assume that $f\in C^{\left\lceil s\right\rceil } $ satisfies \eqref{GrindEQ__3_1_}
for any $0\le k\le\left\lceil s\right\rceil $ and $z\in \mathbb C$. Suppose also that
\begin{equation} \nonumber
\frac{1}{p} =\sigma \left(\frac{1}{r} -\frac{s}{n} \right)+\frac{1}{r}+\frac{b}{n}, ~\frac{1}{r} -\frac{s}{n} >0.
\end{equation}
Then we have
\begin{equation}\label{GrindEQ__3_20_}
\left\| |x|^{-b}f(u)\right\| _{\dot{H}_{p,2}^{s} } \lesssim \left\| u\right\| _{\dot{H}_{r,2}^{s} }^{\sigma +1},
\end{equation}
\begin{equation}\label{GrindEQ__3_21_}
\left\| |x|^{-b}|u|^{\sigma}v\right\| _{L^{p,2}} \lesssim \left\| u\right\| _{\dot{H}_{r,2}^{s} }^{\sigma}\left\| v\right\| _{L^{r,2}}.
\end{equation}
\end{lemma}
\begin{proof}
It follows from Lemma \ref{lem 2.9.} and Lemma \ref{lem 2.10.} (fractional product rule) that
\begin{eqnarray}\begin{split} \label{GrindEQ__3_22_}
\left\| |x|^{-b}f(u)\right\| _{\dot{H}_{p,2}^{s} } &\lesssim \sum _{\left|\alpha \right|=[s]}\left\| D^{\alpha } \left(|x|^{-b}f(u)\right)\right\| _{\dot{H}_{p,2}^{v}}
\lesssim \sum_{\left|\alpha'\right|+\left|\alpha''\right|=[s]}\left\| D^{\alpha'} (|x|^{-b})D^{\alpha''} \left(f(u)\right)\right\| _{\dot{H}_{p,2}^{v}}\\
&\lesssim\sum_{k=0}^{[s]}\left(\left\||x|^{-b}\right\|_{\dot{H}_{p_{k},\infty}^{k+v}}
\left\|f(u)\right\|_{\dot{H}_{r_{k},2}^{[s]-k}}+
\left\||x|^{-b}\right\|_{\dot{H}_{\bar{p}_{k},\infty}^{k}}
\left\|f(u)\right\|_{\dot{H}_{\bar{r}_{k},2}^{s-k}}\right)
,
\end{split}\end{eqnarray}
where $v=s-[s]$ and
\begin{equation}\label{GrindEQ__3_23_}
\frac{1}{p_{k}}=\frac{b+k+v}{n},~\frac{1}{r_{k}}=\frac{1}{p}-\frac{1}{p_{k}},~
\frac{1}{\bar{p}_{k}}=\frac{b+k}{n},~\frac{1}{\bar{r}_{k}}=\frac{1}{p}-\frac{1}{\bar{p}_{k}}.
\end{equation}
Putting
\begin{equation} \label{GrindEQ__3_24_}
\frac{1}{\bar{p}} =\sigma \left(\frac{1}{r} -\frac{s}{n} \right)+\frac{1}{r} ,
\end{equation}
it follows \eqref{GrindEQ__3_23_} that
\begin{equation}\label{GrindEQ__3_25_}
\frac{1}{\bar{p}}=\frac{1}{r_{k}}+\frac{s-([s]-k)}{n}
=\frac{1}{\bar{r}_{k}}+\frac{s-(s-k)}{n}.
\end{equation}
By \eqref{GrindEQ__3_25_} and Lemma \ref{lem 2.8.}, we have the embeddings:
\begin{equation}\label{GrindEQ__3_26_}
\dot{H}_{\bar{p},2}^{s}\subset \dot{H}_{r_{k},2}^{[s]-k},~\dot{H}_{\bar{p},2}^{s}\subset \dot{H}_{\bar{r}_{k},2}^{s-k}.
\end{equation}
Meanwhile, using \eqref{GrindEQ__3_23_}, Remark \ref{rem 3.2.} and Lemma \ref{lem 2.1.}, we have
\begin{equation}\label{GrindEQ__3_27_}
\left\||x|^{-b}\right\|_{\dot{H}_{p_{k},\infty}^{k+v}}<\infty,~
\left\||x|^{-b}\right\|_{\dot{H}_{\bar{p}_{k},\infty}^{k}}<\infty,
\end{equation}
Using \eqref{GrindEQ__3_22_}, \eqref{GrindEQ__3_24_}, \eqref{GrindEQ__3_26_}, \eqref{GrindEQ__3_27_} and Lemma \ref{lem 3.1.}, we have \eqref{GrindEQ__3_20_}. We also have
\begin{equation}\nonumber
\left\| |x|^{-b}|u|^{\sigma}v\right\| _{L^{p,2}}\lesssim \left\| |x|^{-b}\right\| _{L^{n/b,\infty}}\left\||u|^{\sigma}v\right\|_{L^{\bar{p},2}}\lesssim \left\| u\right\| _{L^{a,2}}^{\sigma}\left\| v\right\| _{L^{r,2}}\lesssim\left\| u\right\| _{\dot{H}_{r,2}^{s} }^{\sigma}\left\| v\right\| _{L^{r,2}},
\end{equation}
where $\frac{1}{a}=\frac{1}{r}-\frac{s}{n}$. This completes the proof.
\end{proof}
\begin{remark}\label{rem 3.4.}
\textnormal{If $f(z)$ is a polynomial in $z$ and $\bar{z}$ satisfying $1<\deg (f)=1+\sigma $, we can see that the assumption $\sigma>\left\lceil s\right\rceil -1$ can be removed in Lemma \ref{lem 3.1.} and Lemma \ref{lem 3.3.}.}
\end{remark}
Now we can prove Theorem \ref{thm 1.7.} by using Lemma \ref{lem 2.13.} (Strichartz estimates) and Lemma \ref{lem 3.3.}.
\begin{proof}[{\bf Proof of Theorem \ref{thm 1.7.}}]
Since the proof is standard, we only give the proof of the local well-posedness for later use in the next section. See, for example, Section 4.9 of \cite{C03} or Section 4.3 of \cite{WHHG11}.

If $n\ge 3$, we put $\bar{r}=\frac{2n}{n-2}$, $r=\frac{2n\sigma +2n}{n+2+2\sigma s -2b}$ as in \eqref{GrindEQ__1_12_}. We can easily see that $(\gamma(\bar{r}),\bar{r})$ and $(\gamma(r),r)$ are admissible. We can also see that $\frac{1}{r}>\frac{s}{n}$ is equivalent to $b<1+\frac{n-2s}{2}$. Furthermore, we have
\begin{equation} \label{GrindEQ__3_28_}
\frac{1}{\bar{r}'} =\sigma \left(\frac{1}{r} -\frac{s}{n} \right)+\frac{1}{r} +\frac{b}{n},~\frac{1}{r}>\frac{s}{n}.
\end{equation}

Next, we choose admissible pairs $(\gamma(\bar{r}),\bar{r})$ and $(\gamma(r),r)$ satisfying \eqref{GrindEQ__3_28_} in the case $n\le 2$.
Since $b<n-s$, we can take $\varepsilon>0$ sufficiently small such that $\varepsilon<\min\{n-s-b, \frac{n}{2}\}$.
Putting $r=\frac{\sigma n+n}{\sigma s+n-b-\varepsilon}$, we have $r>2$. We can also see that $\frac{1}{r}>\frac{s}{n}$ is equivalent to $\varepsilon<n-s-b$. In view of \eqref{GrindEQ__3_28_}, we we can see that $\bar{r}>2$ is equivalent to $\varepsilon<\frac{n}{2}$.
Furthermore, we have
\begin{equation} \label{GrindEQ__3_29_}
\frac{1}{\gamma (\bar{r})'} =\frac{\sigma +1}{\gamma (r)},
\end{equation}
for any $n\in \mathbb N$. Let $T>0$ and $M>0$ which will be chosen later. Given $I=[-T,\;T]$, we define
\begin{equation} \nonumber
D=\left\{u\in L^{\gamma (r)} (I,\;H^{s}_{r,2}):\;\left\| u\right\| _{L^{\gamma (r)} \left(I,\;H^{s}_{r,2}\right)} \le M\right\}.
\end{equation}
Putting
\begin{equation} \nonumber
d\left(u,\;v\right)=\;\left\| u-v\right\| _{L^{\gamma(r)} (I,\;L^{r,2})},
\end{equation}
$(D,d)$ is a complete metric space (see \cite{AT21}).
Now we consider the mapping
\begin{equation} \nonumber
G:\;u(t)\to S(t)u_{0} -i\lambda \int _{0}^{t}S(t-\tau)|x|^{-b} f(u(\tau ))d\tau
\equiv u_{L} +u_{NL} ,
\end{equation}
where
\begin{equation} \nonumber
u_{L} =S(t)u_{0} ,~u_{NL} =-i\lambda \int _{0}^{t}S(t-\tau )|x|^{-b}f(u(\tau ))d\tau  .
\end{equation}
Lemma \ref{lem 2.13.} (Strichartz estimates) yields that
\begin{equation}\label{GrindEQ__3_30_}
\left\| u_{NL} \right\| _{L^{\gamma (r)} (I,\;H^{s}_{r,2})} \lesssim\left\| |x|^{-b}f(u)\right\|
_{L^{\gamma (\bar{r})' } (I,\;H^{s}_{\bar{r}',2})},
\end{equation}
\begin{equation}\label{GrindEQ__3_31_}
\left\| Gu-Gv \right\| _{L^{\gamma (r)} (I,\;L^{r,2})} \lesssim\left\| |x|^{-b} \left(f(u)-f(v)\right)\right\|_{L^{\gamma (\bar{r})' } (I,\;L^{\bar{r}',2})}.
\end{equation}
Using the fact
\begin{equation} \label{GrindEQ__3_32_}
\left||x|^{-b} f(u)-|x|^{-b} f(v)\right|\lesssim|x|^{-b} \left(\left|u\right|^{\sigma } +\left|v\right|^{\sigma } \right)|u-v|,
\end{equation}
it follows from \eqref{GrindEQ__3_28_}, \eqref{GrindEQ__3_29_} and Lemma \ref{lem 3.3.} that
\begin{equation} \label{GrindEQ__3_33_}
\left\| |x|^{-b} f(u)\right\| _{L^{\gamma (\bar{r})' } (I,\;H_{\bar{r}'}^{s} )} \lesssim\left\| u\right\| _{L^{\gamma (r)} (I,\;\dot{H}_{r}^{s} )}^{\sigma } \left\| u\right\| _{L^{\gamma (r)} (I,\;H_{r}^{s} )},
\end{equation}
\begin{eqnarray}\begin{split} \label{GrindEQ__3_34_}
&\left\| |x|^{-b}f(u)-|x|^{-b}f(v)\right\| _{L^{\gamma (\bar{r})' } \left(I,\;L^{\bar{r}'} \right)}
\lesssim\left(\left\| u\right\| _{L^{\gamma (r)} (I,\;\dot{H}_{r}^{s} )}^{\sigma } +\left\| u\right\| _{L^{\gamma (r)} (I,\;\dot{H}_{r}^{s} )}^{\sigma } \right)\left\| u-v\right\| _{L^{\gamma (r)} (I,\;L^{r} )}.
\end{split}\end{eqnarray}
By the Strichartz estimates \eqref{GrindEQ__2_5_}, we can see that
$$
\left\|S(t)u_{0} \right\| _{L^{\gamma (r)} ([-T,\;T], \;H^{s}_{r,2})} \to 0,~\textnormal{as}~ T\to 0.
$$
Take $M>0$ such that $CM^{\sigma } \le \frac{1}{4} $ and $T>0$ such that
\begin{equation}\label{GrindEQ__3_35_}
\left\| S(t)u_{0} \right\| _{L^{\gamma (r)}([-T,\;T],\;H^{s}_{r,2} )} \le \frac{M}{4} .
\end{equation}
It follows from \eqref{GrindEQ__3_30_}--\eqref{GrindEQ__3_35_} that
\begin{equation}\label{GrindEQ__3_36_}
\left\| Gu\right\| _{L^{\gamma (r)} (I,\;H^{s}_{r,2})} \le \left\| S(t)u_{0} \right\| _{L^{\gamma (r)}(I,\; H^{s}_{r,2}) } +C\left\| u\right\| _{L^{\gamma (r)}
(I,\;H^{s}_{r,2})}^{\sigma +1} \le \frac{M}{2},
\end{equation}
\begin{equation}\label{GrindEQ__3_37_}
\left\| Gu-Gv\right\| _{L^{\gamma (r)} \left(I,\;L^{r}\right)}
\le 2CM^{\sigma } \left\| u-v\right\| _{L^{\gamma(r)} (I,\;L^{r,2})}
\le \frac{1}{2} \left\|u-v\right\| _{L^{\gamma (r)} (I,\;L^{r,2})} .
\end{equation}
\eqref{GrindEQ__3_36_} and \eqref{GrindEQ__3_37_} imply that $G: (D,d)\to (D,d)$ is a contraction mapping. From the Banach fixed
point theorem, there exists a unique solution $u$ of \eqref{GrindEQ__1_1_} in
$(D,d)$. We can extend the above solution to the maximal solution, using a standard argument and we omit the details. This concludes the proof.
\end{proof}

%%%%%%%%%%%%%%%%%%%%%%%%%%%%%%%%%%%%%%%%%%%%%%%%%%%%%%%%%%%%%%%%%%%%%%%%%%%%%
\section{Continuous dependence}
In this section, we study the continuous dependence of the Cauchy problem for the $H^{s}$-critical INLS equation \eqref{GrindEQ__1_1_} in the standard sense in $H^{s}$, i.e. in the sense that in the sense that the local solution flow is continuous $H^{s} \to H^{s} $. To this end, we establish the estimates of terms $f(u)-f(v)$ and $|x|^{-b}(f(u)-f(v))$ in the fractional Sobolev-Lorentz spaces.
\begin{lemma}\label{lem 4.1.}
Let $p>1$, $0<s<1$ and $\sigma > 1$. Assume that $f\in C^{2} \left(\mathbb C\to \mathbb C\right)$ satisfies
\begin{equation} \label{GrindEQ__4_1_}
\left|f^{(k)} (u)\right|\lesssim|u|^{\sigma +1-k} ,
\end{equation}
for any $0\le k\le 2$ and $u\in \mathbb C$.
Suppose also that
\begin{equation} \label{GrindEQ__4_2_}
\frac{1}{p} =\sigma \left(\frac{1}{r} -\frac{s}{n} \right)+\frac{1}{r} ,~\frac{1}{r} -\frac{s}{n} >0.
\end{equation}
Then we have
\begin{equation}\label{GrindEQ__4_3_}
\left\| f(u)-f(v)\right\| _{\dot{H}_{p,2}^{s} } \lesssim\left(\left\| u\right\| _{\dot{H}_{r,2}^{s} }^{\sigma } +\left\| v\right\| _{\dot{H}_{r,2}^{s} }^{\sigma } \right)\left\| u-v\right\| _{\dot{H}_{r,2}^{s} } .
\end{equation}
\end{lemma}
\begin{proof}
Without loss of generality and for simplicity, we assume that $f$ is a function of a real variable. Putting
\begin{equation} \label{GrindEQ__4_4_}
\frac{1}{a} =\frac{1}{r} -\frac{s}{n},\;\frac{1}{\tilde{p}_{1} } =\frac{\sigma }{a} ,~ \frac{1}{\tilde{p}_{2} } =\frac{1}{\tilde{p}_{1} } +\frac{s}{n} ,~
\frac{1}{2}=\frac{\sigma}{q_{1}}+\frac{1}{q_{2}},~ 2<q_{1},\;q_{2}<\infty,
\end{equation}
it follows from \eqref{GrindEQ__4_2_}, \eqref{GrindEQ__4_3_} and Lemma \ref{lem 2.10.} (fractional product rule) that
\begin{equation}\label{GrindEQ__4_5_}
\left\| f(u)-f(v)\right\| _{\dot{H}_{p,2}^{s} }=\left\| (u-v)\int _{0}^{1}f'\left(v+t(u-v)\right)dt \right\| _{\dot{H}_{p,2}^{s} }\lesssim B_{1} +B_{2},
\end{equation}
where

$$ B_{1}=\left\| \int _{0}^{1}f'\left(v+t(u-v)\right)dt \right\| _{L^{\tilde{p}_{1},q_{1}/\sigma} } \left\| u-v\right\| _{\dot{H}_{r,q_{2}}^{s} },$$
$$ B_{2}=\left\| \int _{0}^{1}f'\left(v+t(u-v)\right)dt \right\| _{\dot{H}_{\tilde{p}_{2},q_{1}/\sigma }^{s} } \left\| u-v\right\| _{L^{a,q_{2}}}.$$
First, we estimate $B_{1}$. We have
\begin{equation} \label{GrindEQ__4_6_}
\left\| \int _{0}^{1}f'\left(v+t(u-v)\right)dt \right\| _{L^{\tilde{p}_{1},q_{1}/\sigma}} \le \int _{0}^{1}\left\| f'\left(v+t(u-v)\right)\right\| _{L^{\tilde{p}_{1},q_{1}/\sigma}} dt .
\end{equation}
It also follows from \eqref{GrindEQ__4_1_} that
\begin{equation} \label{GrindEQ__4_7_}
\left|f'\left(v+t(u-v)\right)\right|\lesssim{\mathop{\max }\limits_{t\in \left[0,\, 1\right]}} \left|v+t(u-v)\right|^{\sigma } \lesssim |u|^{\sigma } +|v|^{\sigma },
\end{equation}
for any $0\le t \le 1$. \eqref{GrindEQ__4_4_}--\eqref{GrindEQ__4_7_} imply that
\begin{eqnarray}\begin{split}\label{GrindEQ__4_8_}
\left\| \int _{0}^{1}f'\left(v+t(u-v)\right)dt\right\| _{L^{\tilde{p}_{1},q_{1}/\sigma} }
\lesssim \left\| u\right\| _{L^{a,q_{1}} }^{\sigma } +\left\| v\right\| _{L^{a,q_{1}} }^{\sigma }
\lesssim \left\| u\right\| _{L^{a,2} }^{\sigma } +\left\| v\right\| _{L^{a,2} }^{\sigma }
\lesssim\left\|u\right\|_{\dot{H}_{r,2}^{s}}^{\sigma}+\left\|v\right\|_{\dot{H}_{r,2}^{s}}^{\sigma },
\end{split}\end{eqnarray}
where the last inequality follows from the embedding $\dot{H}_{r,2}^{s} \subset L^{a, 2} $. Hence we have
\begin{equation} \label{GrindEQ__4_9_}
B_{1} \lesssim \left(\left\| u\right\| _{\dot{H}_{r,2}^{s} }^{\sigma } +\left\| v\right\| _{\dot{H}_{r,2}^{s} }^{\sigma } \right)\left\| u-v\right\| _{\dot{H}_{r,q_{2}}^{s} }
\lesssim\left(\left\| u\right\| _{\dot{H}_{r,2}^{s} }^{\sigma }+\left\| v\right\| _{\dot{H}_{r,2}^{s} }^{\sigma } \right)\left\| u-v\right\| _{\dot{H}_{r,2}^{s} } .
\end{equation}
Next, we estimate $B_{2}$. We have
\begin{equation} \label{GrindEQ__4_10_}
\left\|\int_{0}^{1}f'\left(v+t(u-v)\right)dt \right\|_{\dot{H}_{\tilde{p}_{2},q_{1}/\sigma }^{s} } \le \int _{0}^{1}\left\| f'\left(v+t(u-v)\right)\right\| _{\dot{H}_{\tilde{p}_{2} ,q_{1}/\sigma }^{s} } dt .
\end{equation}
Meanwhile, Lemma \ref{lem 2.11.} (fractional chain rule) and \eqref{GrindEQ__4_1_} yield that
\begin{eqnarray}\begin{split} \label{GrindEQ__4_11_}
\left\| f'\left(v+t(u-v)\right)\right\| _{\dot{H}_{\tilde{p}_{2},q_{1}/\sigma  }^{s} } &\lesssim\left\| f''\left(v+t(u-v)\right)\right\|_{L^{\tilde{p}_{3}, q_{1}/(\sigma-1)} } \left\| v+t(u-v)\right\| _{\dot{H}_{r,q_{1}}^{s} } \\
&\lesssim\left\| \left|v+t(u-v)\right|^{\sigma -1} \right\| _{L^{\tilde{p}_{3},q_{1}/(\sigma-1)} } \left\| v+t(u-v)\right\| _{\dot{H}_{r,q_{1}}^{s} } \\
&\lesssim\left\| v+t(u-v)\right\| _{\dot{H}_{r,2}^{s} }^{\sigma } ,
\end{split}\end{eqnarray}
where $\frac{1}{\tilde{p}_{3} } =\left(\sigma -1\right)\left(\frac{1}{r} -\frac{s}{n} \right)$. We can also see that
\begin{equation} \label{GrindEQ__4_12_}
\left\|v+t(u-v)\right\|_{\dot{H}_{r,2}^{s}}\le{\mathop{\max }\limits_{t\in \left[0,\, 1\right]}} \left\| I^{s} v+t\left(I^{s} u-I^{s} v\right)\right\| _{L^{r,2} } \lesssim\left\| u\right\| _{\dot{H}_{r,2}^{s} } +\left\| v\right\| _{\dot{H}_{r,2}^{s} } .
\end{equation}
\eqref{GrindEQ__4_10_}--\eqref{GrindEQ__4_12_} imply that
\begin{equation} \label{GrindEQ__4_13_}
B_{2}=\left\|\int_{0}^{1}f'\left(v+t(u-v)\right)dt\right\|_{\dot{H}_{\tilde{p}_{2},q_{1}/\sigma }^{s} }\left\| u-v\right\| _{L^{a,q_{2}}}
\lesssim\left(\left\| u\right\| _{\dot{H}_{r,2}^{s} }^{\sigma } +\left\| v\right\| _{\dot{H}_{r,2}^{s} }^{\sigma } \right)\left\| u-v\right\| _{\dot{H}_{r,2}^{s} } .
\end{equation}
In view of \eqref{GrindEQ__4_9_} and \eqref{GrindEQ__4_13_}, we have
\begin{equation} \label{GrindEQ__4_14_}
\left\| f(u)-f(v)\right\| _{\dot{H}_{p,2}^{s} } \lesssim B_{1} +B_{2}
\lesssim\left(\left\| u\right\| _{\dot{H}_{r,2}^{s} }^{\sigma } +\left\| v\right\| _{\dot{H}_{r,2}^{s} }^{\sigma } \right)\left\| u-v\right\| _{\dot{H}_{r,2}^{s} },
\end{equation}
this completes the proof.
\end{proof}

\begin{lemma}\label{lem 4.2.}
Let $p>1$, $s\ge 1$ and $\sigma \ge \left\lceil s\right\rceil $. Assume that $f\in C^{\left\lceil s\right\rceil } \left(\mathbb C\to \mathbb C\right)$ satisfies \eqref{GrindEQ__4_1_}
for any $0\le k\le \left\lceil s\right\rceil $ and $u\in \mathbb C$. Assume further
\begin{equation} \label{GrindEQ__4_15_}
\left|f^{\left(\left\lceil s\right\rceil\right)} (u)-f^{\left(\left\lceil s\right\rceil \right)} (v)\right|\lesssim \left(|u|+|v|\right)^{\sigma -\left\lceil s\right\rceil}|u-v| ,
\end{equation}
for any $u,\;v\in \mathbb C$. Suppose also that \eqref{GrindEQ__4_2_} holds. Then we have \eqref{GrindEQ__4_3_}.
\end{lemma}
\begin{proof}
By Lemma \ref{lem 2.9.}, we have
\begin{equation} \label{GrindEQ__4_16_}
\left\| f(u)-f(v)\right\| _{\dot{H}_{p,2}^{s} } \lesssim\sum _{|\alpha |=[s]}\left\| D^{\alpha } f(u)-D^{\alpha } f(v)\right\| _{\dot{H}_{p,2}^{v} }  ,
\end{equation}
where $v=s-[s]$. It follows from \eqref{GrindEQ__3_5_} that
\begin{eqnarray}\begin{split} \label{GrindEQ__4_17_}
\left\| D^{\alpha } f(u)-D^{\alpha } f(v)\right\| _{\dot{H}_{p,2}^{v} } &=\left\| \sum _{q=1}^{|\alpha |}\sum _{\Lambda _{\alpha }^{q} }C_{\alpha ,\;q} \left(f^{(q)} (u)\prod _{i=1}^{q}D^{\alpha _{i} } u -f^{(q)} (v)\prod _{i=1}^{q}D^{\alpha _{i} } v \right)  \right\| _{\dot{H}_{p,2}^{v} } \\
&\le \sum _{q=1}^{|\alpha |}\sum _{\Lambda _{\alpha }^{q} }C_{\alpha ,\;q} \left(\left\| C_{1} \right\| _{\dot{H}_{p,2}^{v} } +\left\| C_{2} \right\| _{\dot{H}_{p,2}^{v} } \right) ,
\end{split}\end{eqnarray}
where
\begin{equation} \label{GrindEQ__4_18_}
C_{1} =\left(f^{(q)} (u)-f^{(q)} (v)\right)\prod _{i=1}^{q}D^{\alpha _{i} } u ,~C_{2} =f^{(q)} (v)\left(\prod _{i=1}^{q}D^{\alpha _{i} } u -\prod _{i=1}^{q}D^{\alpha _{i} } v \right).
\end{equation}
We only consider the case $s\notin \mathbb N$. The case $s\in \mathbb N$ is treated similarly and more easily. We use the similar argument as in the proof of Lemma \ref{lem 3.1.}. As in the proof of Lemma \ref{lem 3.1.}, putting
\begin{equation} \label{GrindEQ__4_19_}
\frac{1}{a} =\frac{1}{r} -\frac{s}{n} ,~\frac{1}{a_{i} } =\frac{1}{r} -\frac{s-|\alpha _{i}|}{n},
\end{equation}
we have
\begin{equation} \label{GrindEQ__4_20_}
\frac{1}{p} =\frac{\sigma +1-q}{a} +\sum _{i=1}^{q}\frac{1}{a_{i} }  +\frac{v}{n} .
\end{equation}
{\bf Step 1.} First, we estimate $\left\| C_{1} \right\| _{\dot{H}_{p,2}^{v} } $, where $C_{1}$ is given in \eqref{GrindEQ__4_18_}. It follows from \ref{lem 2.10.} (fractional product rule) that
\begin{eqnarray}\begin{split} \label{GrindEQ__4_21_}
\left\| C_{1} \right\| _{\dot{H}_{p,2}^{v} } &\lesssim \left\| f^{(q)} (u)-f^{(q)}(v)\right\| _{\dot{H}_{p_{1},q_{3} }^{v} } \left\| \prod _{i=1}^{q}D^{\alpha _{i} } u \right\| _{L^{r_{1},q_{4}}} +\left\|f^{(q)} (u)-f^{(q)}(v)\right\| _{L^{p_{2},q_{3}} } \left\| \prod _{i=1}^{q}D^{\alpha _{i} } u \right\| _{\dot{H}_{r_{2},q_{4}}^{v} }  \\
&\equiv D_{1} +D_{2},
\end{split}\end{eqnarray}
where $p_{1}$, $p_{2}$, $r_{1}$, $r_{2}$, $q_{3}$ and $q_{4}$ are given in \eqref{GrindEQ__3_9_}--\eqref{GrindEQ__3_11_}.

First, we estimate $D_{1}$. Putting $\frac{1}{p_{6} } =\frac{\sigma +1-q}{a} +\frac{1}{n} $, we can see that
\[0<\frac{1}{p_{6} } \le \sigma \left(\frac{1}{r} -\frac{s}{n} \right)+\frac{1}{n} \le \sigma \left(\frac{1}{r} -\frac{s}{n} \right)+\frac{s}{n} <\sigma \left(\frac{1}{r} -\frac{s}{n} \right)+\frac{1}{r} =\frac{1}{p_{1} } <1,\]
which implies that $1<p_{6} <\infty $ and $\dot{H}_{p_{6},q_{3} }^{1} \subset \dot{H}_{p_{1},q_{3} }^{v} $. Thus we have
\begin{eqnarray}\begin{split} \label{GrindEQ__4_22_}
&\left\| f^{(q)} (u)-f^{(q)} (v)\right\| _{\dot{H}_{p_{1},q_{3}}^{v} } \lesssim\left\| f^{(q)} (u)-f^{(q)} (v)\right\| _{\dot{H}_{p_{6},q_{3}}^{1} }
=\sum _{i=1}^{n}\left\| f^{\left(q+1\right)} (u)\partial _{x_{i} } u-f^{\left(q+1\right)} (v)\partial _{x_{i} } v\right\| _{L^{p_{6},q_{3}}}   \\
&~~~~~~\le \sum _{i=1}^{n}\left\| \left(f^{\left(q+1\right)} (u)-f^{\left(q+1\right)} (v)\right)\partial _{x_{i} } u\right\| _{L^{p_{6},q_{3}}} +\left\| f^{\left(q+1\right)} (v)\left(\partial _{x_{i} } u-\partial _{x_{i} } v\right)\right\| _{L^{p_{6},q_{3}}}   \\
&~~~~~~\lesssim\sum _{i=1}^{n}\left\| f^{\left(q+1\right)} (u)-f^{\left(q+1\right)} (v)\right\| _{L^{p_{7},q_{5} }} \left\| \partial _{x_{i} } u\right\| _{L^{r_{7},q_{1}} } +\left\| f^{\left(q+1\right)} (v)\right\| _{L^{p_{7},q_{5} }} \left\| \partial _{x_{i} } u-\partial _{x_{i} } v\right\| _{L^{r_{7},q_{1}} },
\end{split}\end{eqnarray}
where
\begin{equation} \label{GrindEQ__4_23_}
\frac{1}{p_{7} } =\frac{\sigma -q}{a} ,~\frac{1}{r_{7} } =\frac{1}{r} -\frac{s-1}{n},~\frac{1}{q_{5}}=\frac{\sigma-q}{q_{1}} .
\end{equation}
Using the embedding $\dot{H}_{r,2}^{s} \subset \dot{H}_{r_{7},q_{1} }^{1} $, we have
\begin{equation} \label{GrindEQ__4_24_}
\left\| \partial _{x_{i} } u\right\| _{L^{r_{7},q_{1}} } \lesssim\left\| u\right\| _{\dot{H}_{r_{7},q_{1} }^{1} } \lesssim\left\| u\right\| _{\dot{H}_{r,2}^{s} } ,~ \left\| \partial _{x_{i} } u-\partial _{x_{i} } v\right\| _{L^{r_{7},q_{1}} } \lesssim\left\| u-v\right\| _{\dot{H}_{r_{7},q_{1} }^{1} } \lesssim\left\| u-v\right\| _{\dot{H}_{r,2}^{s} } .
\end{equation}
We also have
\begin{equation} \label{GrindEQ__4_25_}
\left\| f^{\left(q+1\right)} (v)\right\| _{L^{p_{7},q_{5}} }\lesssim \left\| v\right\| _{L^{a,q_{1}} }^{\sigma -q} \lesssim \left\| v\right\| _{\dot{H}_{r,2}^{s} }^{\sigma -q} .
\end{equation}

$\cdot$ If $q=[s]=\left\lceil s\right\rceil -1$, then it follows from \eqref{GrindEQ__4_15_}, \eqref{GrindEQ__4_23_}, Lemma \ref{lem 2.4.} and the embedding $H_{r,2}^{s}\subset L^{a,q_{1}}$ that
\begin{equation} \label{GrindEQ__4_26_}
\left\| f^{\left(\left\lceil s\right\rceil \right)} (u)-f^{\left(\left\lceil s\right\rceil \right)} (v)\right\| _{L^{p_{7},q_{5}} } \lesssim \left(\left\| u\right\| _{\dot{H}_{r,2}^{s} }^{\sigma -\left\lceil s\right\rceil} +\left\| v\right\| _{\dot{H}_{r,2}^{s} }^{\sigma -\left\lceil s\right\rceil} \right)\left\| u-v\right\| _{\dot{H}_{r,2}^{s} }.
\end{equation}
In view of \eqref{GrindEQ__4_22_}, \eqref{GrindEQ__4_24_}--\eqref{GrindEQ__4_26_}, we have
\begin{eqnarray}\begin{split} \label{GrindEQ__4_27_}
\left\| f^{\left([s]\right)} (u)-f^{\left([s]\right)} (v)\right\| _{\dot{H}_{p_{1},q_{3}}^{v} }\lesssim \left(\left\| u\right\| _{\dot{H}_{r,2}^{s} }^{\sigma -[s]} +\left\| v\right\| _{\dot{H}_{r,2}^{s} }^{\sigma -[s]} \right)\left\| u-v\right\| _{\dot{H}_{r,2}^{s} } .
\end{split}\end{eqnarray}
It follows from \eqref{GrindEQ__3_12_} and \eqref{GrindEQ__4_27_} that
\begin{eqnarray}\begin{split} \label{GrindEQ__4_28_}
D_{1} &=\left\| f^{\left([s]\right)} (u)-f^{\left([s]\right)} (v)\right\| _{\dot{H}_{p_{1},q_{3} }^{v} } \left\| \prod _{i=1}^{[s]}D^{\alpha _{i} } u \right\| _{L^{r_{1},q_{4}}}\lesssim\left(\left\| u\right\| _{\dot{H}_{r,2}^{s} }^{\sigma }+\left\| v\right\| _{\dot{H}_{r,2}^{s} }^{\sigma } \right)\left\| u-v\right\| _{\dot{H}_{r,2}^{s} }.
\end{split}\end{eqnarray}

$\cdot$ If $q<[s]=\left\lceil s\right\rceil -1$, then it follows from \eqref{GrindEQ__4_1_} that
\begin{eqnarray}\begin{split} \label{GrindEQ__4_29_}
\left\| f^{\left(q+1\right)} (u)-f^{\left(q+1\right)} (v)\right\| _{L^{p_{7},q_{5}} } &=\left\| (u-v)\int _{0}^{1}f^{\left(q+2\right)} \left(v+t(u-v)\right)dt \right\| _{L^{p_{7},q_{5}} }\\
&\lesssim \left\| (u-v)\int _{0}^{1}\left|v+t(u-v)\right|^{\sigma -q-1} dt \right\| _{L^{p_{7},q_{5}} }\\
&\lesssim \left\| |u|+|v|\right\| _{L^{a,q_{1}}}^{\sigma -q-1} \left\| u-v\right\| _{L^{a,q_{1}}}\\
&\lesssim \left(\left\| u\right\| _{\dot{H}_{r,2}^{s} }^{\sigma -q-1} +\left\| v\right\| _{\dot{H}_{r,2}^{s} }^{\sigma -q-1} \right) \left\| u-v\right\| _{\dot{H}_{r,2}^{s} }.
\end{split}\end{eqnarray}
In view of \eqref{GrindEQ__4_22_}, \eqref{GrindEQ__4_24_}, \eqref{GrindEQ__4_25_} and \eqref{GrindEQ__4_29_}, we have
\begin{equation} \label{GrindEQ__4_30_}
\left\| f^{(q)} (u)-f^{(q)} (v)\right\| _{\dot{H}_{p_{1},q_{3} }^{v} } \lesssim\left(\left\| u\right\| _{\dot{H}_{r,2}^{s} }^{\sigma -q} +\left\| v\right\| _{\dot{H}_{r,2}^{s} }^{\sigma -q} \right)\left\| u-v\right\| _{\dot{H}_{r,2}^{s} } .
\end{equation}
\eqref{GrindEQ__3_12_} and \eqref{GrindEQ__4_30_} yield that
\begin{equation} \label{GrindEQ__4_31_}
D_{1} \lesssim \left(\left\| u\right\| _{\dot{H}_{r,2}^{s} }^{\sigma } +\left\| v\right\| _{\dot{H}_{r,2}^{s} }^{\sigma } \right)\left\| u-v\right\| _{\dot{H}_{r,2}^{s} } .
\end{equation}
Thus, for any $1\le q\le [s]$, we have \eqref{GrindEQ__4_31_}.

Next, we estimate $D_{2}$.
Since $q\le [s]<\left\lceil s\right\rceil $, it follows from \eqref{GrindEQ__4_1_} that
\begin{eqnarray}\begin{split} \label{GrindEQ__4_32_}
\left\| f^{(q)} (u)-f^{(q)} (v)\right\| _{L^{p_{2},q_{3}} } &\lesssim \left\| (u-v)\int _{0}^{1}\left|v+t(u-v)\right|^{\sigma -q} dt \right\| _{L^{p_{2},q_{3}} }  \\
&\lesssim \left\| |u|+|v|\right\| _{L^{a,q_{1}}}^{\sigma -q} \left\| u-v\right\| _{L^{a,q_{1}}}\lesssim \left(\left\| u\right\| _{\dot{H}_{r,2}^{s} }^{\sigma -q} +\left\| v\right\| _{\dot{H}_{r,2}^{s} }^{\sigma -q} \right)\left\| u-v\right\| _{\dot{H}_{r,2}^{s} }.
\end{split}\end{eqnarray}
In view of \eqref{GrindEQ__3_18_} and \eqref{GrindEQ__4_32_}, we have
\begin{equation} \label{GrindEQ__4_33_}
D_{2} =\left\| f^{(q)} (u)-f^{(q)} (v)\right\| _{L^{p_{2},q_{3}} } \left\| \prod _{i=1}^{q}D^{\alpha _{i} } u \right\| _{\dot{H}_{r_{2},q_{4}}^{v} } \lesssim\left(\left\| u\right\| _{\dot{H}_{r,2}^{s} }^{\sigma } +\left\| v\right\| _{\dot{H}_{r,2}^{s} }^{\sigma } \right)\left\| u-v\right\| _{\dot{H}_{r,2}^{s} }.
\end{equation}
In view of \eqref{GrindEQ__4_21_}, \eqref{GrindEQ__4_31_} and \eqref{GrindEQ__4_33_}, we have
\begin{eqnarray}\begin{split} \label{GrindEQ__4_34_}
\left\| C_{1} \right\| _{\dot{H}_{p,2}^{v} }\lesssim D_{1}+D_{2}\lesssim\left(\left\| u\right\| _{\dot{H}_{r,2}^{s} }^{\sigma } +\left\| v\right\| _{\dot{H}_{r,2}^{s} }^{\sigma } \right)\left\| u-v\right\| _{\dot{H}_{r,2}^{s} }.
\end{split}\end{eqnarray}

{\bf Step 2.} Next, we estimate $\left\| C_{2} \right\| _{\dot{H}_{p,2}^{v} }$.
It follows from \eqref{GrindEQ__3_9_}--\eqref{GrindEQ__3_11_} and Lemma \ref{lem 2.10.} (fractional product rule) that
\begin{eqnarray}\begin{split}\label{GrindEQ__4_35_}
\left\| C_{2}\right\| _{\dot{H}_{p,2}^{v} }\lesssim  \left\| f^{(q)} (u)\right\| _{\dot{H}_{p_{1},q_{3} }^{v} } &\left\| \prod _{i=1}^{q}D^{\alpha _{i} } u -\prod _{i=1}^{q}D^{\alpha _{i} } v\right\| _{L^{r_{1},q_{4}}}\\
&+\left\| f^{(q)} (u)\right\| _{L^{p_{2},q_{3}} } \left\| \prod _{i=1}^{q}D^{\alpha _{i} } u -\prod _{i=1}^{q}D^{\alpha _{i} } v \right\| _{\dot{H}_{r_{2},q_{4}}^{v} }.
\end{split}\end{eqnarray}
Repeating the same argument as in the proof of \eqref{GrindEQ__3_12_} and \eqref{GrindEQ__3_18_}, we can easily get
\begin{eqnarray}\begin{split} \label{GrindEQ__4_36_}
\left\| \prod _{i=1}^{q}D^{\alpha _{i} } u -\prod _{i=1}^{q}D^{\alpha _{i} } v\right\| _{L^{r_{1},q_{4}}}&\le \sum _{i=1}^{q}\left\| \prod _{j=1}^{i-1}D^{\alpha _{j} } u \prod _{j=i+1}^{q}D^{\alpha _{j} } v \left(D^{\alpha _{i} } u-D^{\alpha _{i} } v\right)\right\| _{L^{r_{1},q_{4}}}\\
&\lesssim\left(\left\| u\right\| _{\dot{H}_{r,2}^{s} }^{q-1} +\left\| v\right\| _{\dot{H}_{r,2}^{s} }^{q-1} \right)\left\| u-v\right\| _{\dot{H}_{r,2}^{s} },
\end{split}\end{eqnarray}
\begin{eqnarray}\begin{split} \label{GrindEQ__4_37_}
\left\| \prod _{i=1}^{q}D^{\alpha _{i} } u -\prod _{i=1}^{q}D^{\alpha _{i} } v \right\| _{\dot{H}_{r_{2},q_{4}}^{v} }& \le \sum _{i=1}^{q}\left\| \prod _{j=1}^{i-1}D^{\alpha _{j} } u \prod _{j=i+1}^{q}D^{\alpha _{j} } v \left(D^{\alpha _{i} } u-D^{\alpha _{i} } v\right)\right\| _{\dot{H}_{r_{2},q_{4}}^{v} } \\
&\lesssim\left(\left\| u\right\| _{\dot{H}_{r,2}^{s} }^{q-1} +\left\| v\right\| _{\dot{H}_{r,2}^{s} }^{q-1} \right)\left\| u-v\right\| _{\dot{H}_{r,2}^{s} },
\end{split}\end{eqnarray}
whose proofs will be omitted.
In view of \eqref{GrindEQ__4_35_}--\eqref{GrindEQ__4_37_}, \eqref{GrindEQ__3_14_} and \eqref{GrindEQ__3_16_}, we have
\begin{equation} \label{GrindEQ__4_38_}
\left\| C_{2} \right\| _{\dot{H}_{p,2}^{v} } \lesssim\left(\left\| u\right\| _{\dot{H}_{r,2}^{s} }^{\sigma } +\left\| v\right\| _{\dot{H}_{r,2}^{s} }^{\sigma } \right)\left\| u-v\right\| _{\dot{H}_{r,2}^{s} } .
\end{equation}
Using \eqref{GrindEQ__4_16_}, \eqref{GrindEQ__4_17_}, \eqref{GrindEQ__4_34_} and \eqref{GrindEQ__4_38_}, we can get the desired result.
\end{proof}
\begin{lemma}\label{lem 4.3.}
Let $s>0$ and $f\left(z\right)$ be a polynomial in $z$ and $\bar{z}$ satisfying $1<\deg \left(f\right)=1+\sigma $. Suppose also that \eqref{GrindEQ__4_2_} holds. Then we have \eqref{GrindEQ__4_3_}.
\end{lemma}
\begin{proof}
The result in the case $0<s<1$ and $\sigma>1$ follows from Lemma \ref{lem 4.1.}. The case $0<s<1$ and $\sigma=1$ is trivial. In the case $1\le s<\frac{n}{2}$ and $\sigma \ge \left\lceil s\right\rceil $, the result follows from Lemma \ref{lem 4.2.}. Let us consider the case $1\le s<\frac{n}{2}$ and $\sigma \le \left\lceil s\right\rceil -1$. Obviously, we have
\begin{equation} \label{GrindEQ__4_39_}
\left|f^{(k)} (u)\right|\lesssim|u|^{\sigma +1-k} ,
\end{equation}
for any $0\le k\le \sigma +1$ and
\begin{equation} \label{GrindEQ__4_40_}
\left|f^{(k)} (u)\right|=0,
\end{equation}
for any $k>\sigma +1$. Using \eqref{GrindEQ__4_39_} and \eqref{GrindEQ__4_40_}, we have
\begin{equation} \label{GrindEQ__4_41_}
\left|f^{(q)} (u)-f^{(q)} (v)\right|\;\lesssim\left|u-v\right|\left(|u|^{\sigma -q} +|v|^{\sigma -q} \right),
\end{equation}
for any $q\le \sigma $ and
\begin{equation} \label{GrindEQ__4_42_}
\left|f^{(q)} (u)-f^{(q)} (v)\right|=0,
\end{equation}
for any $q>\sigma $. Using \eqref{GrindEQ__4_39_}--\eqref{GrindEQ__4_42_} and the same argument as in the proof of Lemma \ref{lem 4.2.}, we can easily get the desired result.
\end{proof}

Using Lemmas \ref{lem 4.1.}--\ref{lem 4.3.} and the same argument as in the proof of Lemma \ref{lem 3.3.}, we have the following result, whose proof will be omitted.
\begin{lemma}\label{lem 4.4.}
Let $1<p,\;r<\infty $, $b>0$, $s\ge 0$, $b+s<n$ and $\sigma>\max \left\{0,\;\left\lceil
s\right\rceil -1\right\}$.
Assume that one of the following conditions is satisfied:
\begin{itemize}
  \item $f\left(z\right)$ is not a polynomial in $z$ and $\bar{z}$ satisfying $1<\deg (f)=1+\sigma$.
  \item $0<s<1$, $\sigma > 1$ and $f\in C^{2} \left(\mathbb C\to \mathbb C\right)$ satisfies \eqref{GrindEQ__1_3_} for any $0\le k\le 2$ and $u\in \mathbb C$
  \item $s\ge 1$, $\sigma \ge \left\lceil s\right\rceil $ and $f\in C^{\left\lceil s\right\rceil } \left(\mathbb C\to \mathbb C\right)$ satisfies \eqref{GrindEQ__1_3_}
      for any $0\le k\le \left\lceil s\right\rceil $ and $u\in \mathbb C$. Furthermore,
      \begin{equation} \nonumber
      \left|f^{\left(\left\lceil s\right\rceil\right)} (u)-f^{\left(\left\lceil s\right\rceil \right)} (v)\right|\lesssim \left(|u|+|v|\right)^{\sigma -\left\lceil s\right\rceil}|u-v| ,
      \end{equation}
     for any $u,\;v\in \mathbb C$.
\end{itemize}
 Suppose also that
\begin{equation} \nonumber
\frac{1}{p} =\sigma \left(\frac{1}{r} -\frac{s}{n} \right)+\frac{1}{r}+\frac{b}{n}, ~\frac{1}{r} -\frac{s}{n} >0.
\end{equation}
Then we have
\begin{equation}\label{GrindEQ__4_43_}
\left\| |x|^{-b}f(u)\right\| _{\dot{H}_{p,2}^{s} }\lesssim\left(\left\| u\right\| _{\dot{H}_{r,2}^{s} }^{\sigma } +\left\| v\right\| _{\dot{H}_{r,2}^{s} }^{\sigma } \right)\left\| u-v\right\| _{\dot{H}_{r,2}^{s} } .
\end{equation}
\end{lemma}
\begin{proof}[{\bf Proof of Theorem \ref{thm 1.10.}}]
Since $u$, $u_{m} $ satisfy the following integral equations:
\[u(t)=S(t)u_{0}-i\lambda\int_{0}^{t}S(t-\tau)|x|^{-b}f\left(u(\tau)\right)d\tau  ,\]
\[u_{m} (t)=S(t)u_{0}^{m} -i\lambda \int _{0}^{t}S(t-\tau) |x|^{-b} f\left(u_{m} (\tau)\right)d\tau  ,\]
respectively, we have
\begin{equation} \label{GrindEQ__4_44_}
u_{m} (t)-u(t)=S(t) \left(u_{0}^{m} -u_{0} \right)-i\lambda \int _{0}^{t}S(t-\tau) |x|^{-b} \left(f\left(u_{m} (\tau)\right)-f\left(u(\tau)\right)\right)d\tau.
\end{equation}
By Lemma \ref{lem 2.13.} (Strichartz estimates), for any admissible pair $(\gamma(p),p)$, we have
\begin{equation}\label{GrindEQ__4_45_}
\left\|S(t) u_{0}^{m}\right\|_{L^{\gamma(p)}([-T,\;T],\; H_{p,2}^{s})}\le \left\|S(t) u_{0}\right\|_{L^{\gamma(p)}([-T,\;T],\; H_{p,2}^{s})}+C\left\|u_{0}^{m}-u_{0}\right\|_{H^{s}}.
\end{equation}
Hence, if we take $T>0$ sufficiently small such that $\left\|S(t)u_{0}\right\|_{L^{\gamma(p)}([-T,\;T],\; H_{p,2}^{s})}<\frac{M}{4}$, then we also have
$$
\left\|S(t)u_{0}^{m}\right\|_{L^{\gamma(p)}([-T,\;T],\; H_{p,2}^{s})}<\frac{M}{2},
$$
for sufficiently large $m$.
By using the argument in the proof of Theorem \ref{thm 1.7.}, we can construct solutions $u$ and $u_{m}$ ($m$ sufficiently large) in the set $(D,d)$ given in Theorem \ref{thm 1.7.}, which implies that $T<T_{\max}(u_{0}),\;T_{\min}(u_{0})$ and $T<T_{\max}(u_{0}^{m}),\;T_{\min}(u_{0}^{m})$ for $m$ large enough. Furthermore, by Lemma \ref{lem 2.13.} (Strichartz estimates), we can see that given any admissible pair $(\gamma(p),p)$, we have
$$
\max\left\{\left\|u\right\|_{L^{\gamma(p)}([-T,\;T],\; H_{p,2}^{s})},\;\left\|u_{m}\right\|_{L^{\gamma(p)}([-T,\;T],\; H_{p,2}^{s})}\right\}\le CM,
$$
for all sufficiently large $m$. It follows from \eqref{GrindEQ__4_44_} and Lemma \ref{lem 2.13.} (Strichartz estimates) that
\begin{equation} \label{GrindEQ__4_46_}
\left\| u_{m} -u\right\| _{L^{\gamma(p)}([-T,\;T],\; H_{p,2}^{s})} \lesssim\left\| \phi _{m} -\phi \right\| _{H^{s} } +\left\| |x|^{-b} f\left(u_{m} \right)-|x|^{-b} f(u)\right\| _{L^{\gamma(\bar{r})'}([-T,\;T],\; H_{\bar{r}',2}^{s})},
\end{equation}
where we choose $\bar{r}$ as in the proof of Theorem \ref{thm 1.7.}.
Using \eqref{GrindEQ__3_34_}, \eqref{GrindEQ__4_46_}, Lemma \ref{lem 4.4.} and the same argument as in the proof of Theorem \ref{thm 1.7.}, we have
\begin{eqnarray}\begin{split} \label{GrindEQ__4_47_}
&\left\| u_{m} -u\right\| _{L^{\gamma(r)}([-T,\;T],\; H_{r,2}^{s})} \\
&~~~\le C\left\|u_{0}^{m}-u_{0}\right\|_{H^{s}}+2CM^{\sigma}\left(\left\| u_{m}\right\| _{L^{\gamma (r)} (I,\;H_{r}^{s} )}^{\sigma } +\left\| u\right\| _{L^{\gamma (r)} (I,\;H_{r}^{s} )}^{\sigma } \right)\left\| u_{m} -u\right\| _{L^{\gamma(r)}([-T,\;T],\; H_{r,2}^{s})}\\
&~~~\le C\left\|u_{0}^{m}-u_{0}\right\|_{H^{s}}+\frac{1}{2}\left\| u_{m} -u\right\| _{L^{\gamma(r)}([-T,\;T],\; H_{r,2}^{s})},
\end{split}\end{eqnarray}
which implies that as $m\to \infty $,
\[\left\| u_{m} -u\right\| _{L^{\gamma(r)}([-T,\;T],\; H_{r,2}^{s})} \le C\left\| u_{0}^{m} -u_{0} \right\| _{H^{s} } \to 0.\]
By Lemma \ref{lem 2.13.} (Strichartz estimates), we conclude that for all admissible pair $(\gamma(p),p)$,  we have
$$
\left\| u_{m} -u\right\| _{L^{\gamma(p)}([-T,\;T],\; H_{p,2}^{s})}\to 0,~\textrm{as}~m\to 0,
$$ and the continuous dependence is locally Lipschitz.
Using a standard compact argument (see e.g. Subsection 3.2.3 of \cite{DYC13}), we can get the desired result.
\end{proof}
%%%%%%%%%%%%%%%%%%%%%%%%%%%%%%%%%%%%%%%%%%%%%%%%%%%%%%%%%%%%%%%%%%%%%%%%%%%%%
\section{Blow-up}
In this section, we prove Theorem \ref{thm 1.13.}.
First, we recall the virial estimates related to the focusing INLS equation \eqref{GrindEQ__1_14_}. Given a real valued function $a$, we define the virial potential by
\begin{equation}\label{GrindEQ__5_1_}
V_{a}(t):=\int_{\mathbb R^{n}} a(x)\left|u(t,x)\right|^{2}dx.
\end{equation}
\begin{lemma}[Standard virial identity, \cite{D18}]\label{lem 5.1.}
Let $u_{0}\in H^{1}$ be such that $|x|u_{0}\in L^2$ and $u:I\times\mathbb R^{n}\to \mathbb C$ the corresponding solution to the focusing INLS equation \eqref{GrindEQ__1_14_}. Then, $|x|u\in C\left(I,\;L^2\right)$. Moreover, for any $t\in I$,
\begin{equation}\label{GrindEQ__5_2_}
\frac{d^2}{dt^2}\left\|xu(t)\right\|_{L^{2}}^{2}=8 \left\|u(t)\right\| _{\dot{H}^{1}}^{2} -\frac{4(n\sigma+2b)
}{\sigma+2} \int{|x|^{-b} \left|u(t,x)\right|^{\sigma+2}dx.}
\end{equation}
\end{lemma}
In order to recall the localized virial estimates, we introduce a function $\theta:[0,\infty)\to [0,\infty)$ satisfying
\begin{equation}\label{GrindEQ__5_3_}
\theta (r)=\left\{\begin{array}{l} {r^2,~\textrm{if}\;0\le r\le 1,}
\\ {\textrm{0},~\textrm{if}\;r\ge 2,} \end{array}\right.
\textrm{and}~~\theta''(r)\le 2 ~~\textrm{for}~~r\ge 0.
\end{equation}
For $R>1$, we define the radial function on $\mathbb R^{n}$:
\begin{equation}\label{GrindEQ__5_4_}
\varphi_{R}(x)=\varphi_{R}(r):=R^{2}\theta(r/R),\;r=|x|.
\end{equation}
The authors in \cite{DK21} proved the following localized virial estimate for $\sigma_{0}<\sigma<\sigma_{1}$, but the the proof and so the result are still valid for $0<\sigma\le \sigma_{1}$.
\begin{lemma}[Localized Virial Estimate, \cite{DK21}]\label{lem 5.2.}
Let $n \ge 3$, $0<b<2$, $R>1$, $0<\sigma\le \sigma_{1}$ and $\varphi_{R}$ be as in \eqref{GrindEQ__5_4_}. Let $u:I\times\mathbb R^{n}\to \mathbb C$ be a solution to the focusing INLS equation \eqref{GrindEQ__1_14_}. Then for any $t\in I$,
\begin{equation} \label{GrindEQ__5_5_}
\frac{d^2}{dt^2}V_{\varphi_{R}}(t)\le 8 \left\|u(t)\right\| _{\dot{H}^{1}}^{2} -\frac{4(n\sigma+2b)}{\sigma+2} \int{|x|^{-b} \left|u(t,x)\right|^{\sigma +2}dx}+CR^{-2}+CR^{-b}\left\|u(t)\right\|_{H^{1}}^{\sigma+2}.
\end{equation}
\end{lemma}
Next, we recall the sharp Hardy-Sobolev embedding inequality associated with the focusing energy-critical INLS equation \eqref{GrindEQ__1_14_}. For example, see \cite{KP04} or Theorem 15.2.2 in \cite{GM13}.
\begin{lemma}[Sharp Hardy-Sobolev embedding inequality]\label{lem 5.3.}
Let $n\ge 3$, $0< b< 2$ and $\sigma_{1}=\frac{4-2b}{n-2}$. Then we have
\begin{equation}\label{GrindEQ__5_6_}
\left(\int_{\mathbb R^{n}}{|x|^{-b}\left|f\right|^{\sigma_{1}+2}dx}\right)^{\frac{1}{\sigma_{1}+2}}
\le C_{HS} \left\|f\right\|_{\dot{H}^{1}},
\end{equation}
for all $f\in \dot{H}^{1}$, where the sharp Hardy-Sobolev constant $C_{HS}$ defined by
\begin{equation}\label{GrindEQ__5_7_}
C_{HS}(b,c)=\inf_{f\in \dot{H}^{1}\setminus\left\{0\right\}}
{\frac{\left\|f\right\|_{\dot{H}^{1}}}
{\left(\int{|x|^{-b}\left|f\right|^{\sigma_{1}+2}dx}\right)
^{\frac{1}{\sigma_{1}+2}}}}.
\end{equation}
is attained by function $W_{b}(x)$ given in \eqref{GrindEQ__1_16_}.
\end{lemma}
Lemma 2.2 in \cite{KP04} also shows that $W_{b}$ solves the equation
\begin{equation}\nonumber
\Delta W_{b}+|x|^{-b}\left|W_{b}\right|^{\sigma_{1}}W_{b}=0,
\end{equation}
and satisfies
\begin{equation}\label{GrindEQ__5_8_}
\left\|W_{b}\right\|_{\dot{H}^{1}}^{2}=\int{|x|^{-b} W_{b}^{\sigma_{1}+2}dx}.
\end{equation}
Hence, we have
\begin{equation}\label{GrindEQ__5_9_}
\left\|W_{b}\right\|_{\dot{H}^{1}}^{2}=\int{|x|^{-b} W_{b}^{\sigma_{1}+2}dx}=[C_{HS}]^{-\frac{2(n-b)}{2-b}},~\left\|W_{b}\right\|_{\dot{H}^{1}}^{\sigma_{1}}
=[C_{HS}]^{-(\sigma_{1}+2)},
\end{equation}
\begin{equation} \label{GrindEQ__5_10_}
E\left(W_{b}\right)=\frac{1}{2} \left\|
 W_{b}\right\|_{\dot{H}^{1}}^{2}-\frac{1}{\sigma_{1}+2} \int{|x|^{-b} \left|W_{b}\right|^{\sigma_{1}+2}dx}=\frac{2-b}{2(n-b)}[C_{HS}]^{-\frac{2(n-b)}{2-b}}.
\end{equation}

We are ready to prove Theorem \ref{thm 1.13.}.
\begin{proof}[{\bf Proof of Theorem \ref{thm 1.13.}}] We divide the study in two cases: $E(u_{0})<0$ and $E(u_{0})\ge0$

\textbf{Case 1.} We consider the case $E(u_{0})<0$.

Let us prove the first part. If $T_{*},~T^{*}<\infty$, then we are done. If $T^{*}=\infty$, then we prove that there exists $t_{n}\to \infty$ such that $\left\|u(t_{n})\right\|_{\dot{H}^{1}}\to \infty$ as $n\to \infty$. Assume by contradiction that it doesn't hold, i.e. $\sup_{t\in [0,\infty)}{\left\|u(t)\right\|_{\dot{H}^{1}}}\le M_{0}$ for some $M_{0}>0$. By the conservation of mass, we have
$$
\sup_{t\in [0,\infty)}{\left\|u(t)\right\|_{H^{1}}}\le M_{1},
$$
for some $M_{1}>0$. Using the localized virial estimate \eqref{GrindEQ__5_5_} and the conservation of energy, we have
\begin{eqnarray}\begin{split} \nonumber
\frac{d^2}{dt^2}V_{\varphi_{R}}(t)\le &8 \left\|u(t)\right\| _{\dot{H}^{1}}^{2} -\frac{4(n\sigma_{1}+2b)}{\sigma_{1}+2} \int{|x|^{-b} \left|u(t,x)\right|^{\sigma_{1} +2}dx}
+CR^{-2}+CR^{-b}\left\|u(t)\right\|_{H^{1}}^{\sigma_{1}+2}\\
=&4(n\sigma_{1}+2b)E\left(u_{0}\right)-2(n\sigma_{1}-4+2b)
\left\|u\right\|_{\dot{H}^{1}}^{2}+CR^{-2}+CR^{-b}M_{1}^{\sigma+2},
\end{split}\end{eqnarray}
for any $t\in [0,\infty)$. Since $n\sigma_{1}-4+2b>0$, we take $R>1$ large enough to have that
\begin{equation}\nonumber
\frac{d^2}{dt^2}V_{\varphi_{R}}(t)\le2(n\sigma_{1}+2b)E\left(u_{0}\right)<0,
\end{equation}
for any $t\in [0,\infty)$.
Integrating this estimate, there exists $t_{0}>0$ sufficiently large such that $V_{\varphi_{R}}(t_{0})<0$ which is impossible. Similarly, if $T_{*}=\infty$, then we can prove that there exists $t_{n}\to -\infty$ such that $\left\|u(t_{n})\right\|_{\dot{H}^{1}}\to \infty$ as $n\to \infty$. This completes the proof of the first part.

Next, we prove the second part.

$\cdot$ \textbf{Finite-variance data.} Let us assume in addition that $xu_{0}\in L^{2}$. Applying the standard virial identity \eqref{GrindEQ__5_2_} and the conservation of energy, we have for any $t\in (-T_{*},\;T^{*})$,
\begin{eqnarray}\begin{split}\nonumber
\frac{d^2}{dt^2}\left\|xu(t)\right\|_{L^{2}}^{2}&=8 \left\|u(t)\right\| _{\dot{H}^{1}}^{2} -\frac{4(n\sigma_{1}+2b)}{\sigma_{1}+2} \int{|x|^{-b} \left|u(t,x)\right|^{\sigma_{1}+2}dx}\\
&=4(n\sigma_{1}+2b)E\left(u(t)\right)-2(n\sigma_{1}-4+2b)\left\|u\right\|_{\dot{H}^{1}}^{2}<0,
\end{split}\end{eqnarray}
where we used the fact $n\sigma_{1}-4+2b>0$. By the classical argument of Glassey \cite{G77}, it follows that the solution $u$ blows up in finite time.

$\cdot$ \textbf{Radially symmetric data.} Let us assume in addition that $u_{0}$ is radially symmetric.
Using Lemma 3.4 of \cite{D18} and the conservation of energy, we have for any $\varepsilon>0$ and any $t$ in the existence time,
\begin{eqnarray}\begin{split} \nonumber
\frac{d^2}{dt^2}V_{\varphi_{R}}(t)\le & 8 \left\|u(t)\right\| _{\dot{H}^{1}}^{2} -\frac{4(n\sigma_{1}+2b)}{\sigma_{1}+2} \int{|x|^{-b} \left|u(t,x)\right|^{\sigma_{1}+2}dx}\\
&+O\left(R^{-2}+\varepsilon^{-\frac{\sigma_{1}}{4-\sigma_{1}}}
R^{-{\frac{2\left[(n-1)\sigma_{1}+2b\right]}
{4-\sigma_{1}}}}+\varepsilon\left\|u(t)\right\| _{\dot{H}^{1}}^{2}\right)\\
=&4(n\sigma_{1}+2b)E\left(u_{0}\right)-2(n\sigma_{1}-4+2b)
\left\|u\right\|_{\dot{H}^{1}}^{2}\\
&+O\left(R^{-2}+\varepsilon^{-\frac{\sigma_{1}}{4-\sigma_{1}}}
R^{-{\frac{2\left[(n-1)\sigma_{1}+2b\right]}
{4-\sigma_{1}}}}+\varepsilon\left\|u(t)\right\| _{\dot{H}^{1}}^{2}\right),
\end{split}\end{eqnarray}
Since $n\sigma_{1}-4+2b>0$, we take $\varepsilon>0$ small enough and $R>1$ large enough depending on $\varepsilon$ to have that
\begin{equation}\nonumber
\frac{d^2}{dt^2}V_{\varphi_{R}}(t)\le2(n\sigma_{1}+2b)E\left(u_{0}\right)<0,
\end{equation}
for any $t\in (-T_{*},\;T^{*})$. This shows that the solution $u$ must blow up in finite time.

$\cdot$ \textbf{Cylindrically symmetric data.} Let us assume in addition that $b\ge 4-n$ and $u_{0}\in \Sigma_{n}$. We define the radial function on $\mathbb R^{n-1}$:
\begin{equation}\label{GrindEQ__5_11_}
\psi_{R}(y)=\psi_{R}(r):=R^{2}\theta(r/R),\;r=|y|,
\end{equation}
where $R>1$ and $\theta:[0,\infty)\to [0,\infty)$ is given in \eqref{GrindEQ__5_3_}.
 We define the function on $\mathbb R^{n}$:
\begin{equation}\label{GrindEQ__5_12_}
\phi_{R}(x)=\phi_{R}(y,x_{n}):=\psi_{R}(y)+x_{n}^{2}.
\end{equation}
It follows from (4.12) of \cite{DK21} that
\begin{eqnarray}\begin{split}\nonumber
\frac{d^2}{dt^2}V_{\phi_{R}}(t)\le 4(n\sigma_{1}+2b)&E\left(u(t)\right)-2(n\sigma_{1}-4+2b)\left\|u\right\|_{\dot{H}^{1}}^{2}\\
&+CR^{-2}+\left\{\begin{array}{l} {CR^{-\frac{n-2}{2}-b}\left\|u(t)\right\|_{\dot{H}^1}^{2},~\textrm{if}\;\sigma_{1}=2,}
\\{CR^{-\frac{(n-2)\sigma_{1}}{4}-b}\left\|u(t)\right\|_{\dot{H}^1}^{2}
+CR^{-\frac{(n-2)\sigma_{1}}{4}-b},~\textrm{if}\;\sigma_{1}<2,} \end{array}\right.
\end{split}\end{eqnarray}
for any $t\in (-T_{*},\;T^{*})$. Note that $\sigma_{1}\le 2$ is equivalent to $b\ge 4-n$.
Since $n\sigma_{1}-4+2b>0$, we take $R>1$ large enough to have that
$$
\frac{d^2}{dt^2}V_{\phi_{R}}(t)\le2(n\sigma_{1}+2b)E\left(u_{0}\right)<0,
$$
for any $t\in (-T_{*},\;T^{*})$. This shows that the solution $u$ blows up in finite time.

\textbf{Case 2.} We consider the case $E(u_{0})\ge0$. Let us assume that $E(u_{0})<E(W_{b})$ and $\left\|u_{0}\right\|_{\dot{H}^{1}}>\left\|W_{b}\right\|_{\dot{H}^{1}}$.
By definition of the energy and the sharp Hardy-Sobolev inequality \eqref{GrindEQ__5_6_}, we have
\begin{eqnarray}\begin{split}\nonumber
E\left(u(t)\right)&=\frac{1}{2} \left\| u(t)\right\| _{\dot{H}^{1}}^{2} -\frac{1}{\sigma_{1}+2} \int{|x|^{-b} \left|u(t,x)\right|^{\sigma_{1}+2}dx}\\
&\ge \frac{1}{2} \left\| u(t)\right\| _{\dot{H}^{1}}^{2} -\frac{[C_{HS}]^{\sigma_{1}+2}}{\sigma_{1}+2}\left\| u(t)\right\| _{\dot{H}^{1}}^{\sigma_{1}+2}=:g\left(\left\| u(t)\right\| _{\dot{H}^{1}}\right),
\end{split}\end{eqnarray}
where
\begin{equation}\label{GrindEQ__5_13_}
g(y)=\frac{1}{2}y^2-\frac{[C_{HS}]^{\sigma_{1}+2}}{\sigma_{1}+2}y^{\sigma_{1}+2}.
\end{equation}
Lemma \ref{lem 5.3.} shows that
\[g\left(\left\|W_{b}\right\|_{\dot{H}^{1}}\right)=E(W_{b}).\]
By the conservation of energy and the assumption $E(u_{0})<E(W_{b})$, we can see that
\[g\left(\left\| u(t)\right\| _{\dot{H}^{1}}\right)\le E\left(u(t)\right)=E\left(u_{0}\right)<E(W_{b}).\]
By the assumption $\left\|u_{0}\right\|_{\dot{H}^{1}}>\left\|W_{b}\right\|_{\dot{H}^{1}}$ and the continuity argument, we have
\begin{equation}\label{GrindEQ__5_14_}
\left\|u(t)\right\|_{\dot{H}^{1}}>\left\|W_{b}\right\|_{\dot{H}^{1}},
\end{equation}
for any $t$ as long as the solution exists. We next improve \eqref{GrindEQ__5_14_} as follows. Pick $\delta>0$ small enough such that
\begin{equation}\label{GrindEQ__5_15_}
E(u_{0})\le(1-\delta)E(W_{b}),
\end{equation}
which implies that
\begin{equation}\label{GrindEQ__5_16_}
g\left(\left\| u(t)\right\| _{\dot{H}^{1}}\right)\le(1-\delta)E(W_{b}).
\end{equation}
Using \eqref{GrindEQ__5_9_}, \eqref{GrindEQ__5_10_} and \eqref{GrindEQ__5_16_}, we have
\begin{equation}\nonumber
\frac{n-b}{2-b}\left(\frac{\left\| u(t)\right\| _{\dot{H}^{1}}}{\left\| W_{b}\right\| _{\dot{H}^{1}}}\right)^{2}-\frac{n-2}{2-b}\left(\frac{\left\| u(t)\right\| _{\dot{H}^{1}}}{\left\| W_{b}\right\| _{\dot{H}^{1}}}\right)^{\sigma_{1}+2}\le 1-\delta.
\end{equation}
The continuity argument shows that there exits $\delta'>0$ depending on $\delta$ such that
\begin{equation}\label{GrindEQ__5_17_}
\frac{\left\| u(t)\right\| _{\dot{H}^{1}}}{\left\| W_{b}\right\| _{\dot{H}^{1}}}\ge1+\delta'.
\end{equation}
Then we can take $\varepsilon>0$ small enough such that
\begin{equation}\label{GrindEQ__5_18_}
8\left\| u(t)\right\| _{\dot{H}^{1}}^{2} -\frac{4(n\sigma_{1}+2b)}{\sigma_{1}+2}  \int{|x|^{-b} \left|u(t,x)\right|^{\sigma_{1}+2}dx}+\varepsilon\left\| u(t)\right\| _{\dot{H}^{1}}^{2}\le -c<0,
\end{equation}
for any $t$ in the existence time. Indeed, by the definition of energy, we have
\[\textrm{LHS\eqref{GrindEQ__5_18_}}=4(n\sigma_{1}+2b)E\left(u(t)\right)+(8+\varepsilon-2n\sigma_{1}-4b)
\left\| u(t)\right\| _{\dot{H}^{1}}^{2}.\]
Using the conservation of energy, \eqref{GrindEQ__5_9_}, \eqref{GrindEQ__5_10_}, \eqref{GrindEQ__5_15_} and \eqref{GrindEQ__5_17_}, we have
\begin{eqnarray}\begin{split}\nonumber
\textrm{LHS\eqref{GrindEQ__5_18_}}&\le4(1-\delta)(n\sigma_{1}+2b)E(W_{b})+(8+\varepsilon-2n\sigma_{1}-4b)
(1-\delta')^{2}\left\| W_{b}\right\| _{\dot{H}^{1}}^{2}\\
&=\left\|W_{b}\right\|_{\dot{H}^{1}}^{2}\left[\frac{8(2-b)}{n-2}
\left(1-\delta-(1+\delta')^{2}\right)+\varepsilon(1+\delta')^{2}\right].
\end{split}\end{eqnarray}
Hence, by taking $\varepsilon>0$ small enough, we can get \eqref{GrindEQ__5_18_}.

Let us prove the first part. Let us consider only positive times. The one for negative times is similar. If $T^{*}<\infty$, then we are done. If $T^{*}=\infty$, then we prove that there exists $t_{n}\to \infty$ such that $\left\|u(t_{n})\right\|_{\dot{H}^{1}}\to \infty$ as $n\to \infty$. Assume by contradiction that it doesn't hold. Then, as in Case 1, we have
$\sup_{t\in [0,\infty)}{\left\|u(t)\right\|_{H^{1}}}\le M_{1}$
for some $M_{1}>0$. Using the localized virial estimate \eqref{GrindEQ__5_5_} and the conservation of energy, we have
\begin{equation} \nonumber
\frac{d^2}{dt^2}V_{\varphi_{R}}(t)\le 8 \left\|u(t)\right\| _{\dot{H}^{1}}^{2} -\frac{4(n\sigma_{1}+2b)}{\sigma_{1}+2} \int{|x|^{-b} \left|u(t,x)\right|^{\sigma_{1} +2}dx}+CR^{-2}+CR^{-b}M_{1}^{\sigma_{1}+2},
\end{equation}
for any $t\in [0,\infty)$. Taking $R>1$ large enough, \eqref{GrindEQ__5_18_} implies
$$
\frac{d^2}{dt^2}V_{\varphi_{R}}(t)\le-c/2<0,
$$
for any $t\in [0,\infty)$, which is impossible. This completes the proof of the first part.

Next, we prove the second part.

$\cdot$ \textbf{Finite-variance data.}
Let us assume in addition that $xu_{0}\in L^{2}$. Using the standard virial identity \eqref{GrindEQ__5_2_} and \eqref{GrindEQ__5_18_}, we have
\begin{equation}\nonumber
\frac{d^2}{dt^2}\left\|xu(t)\right\|_{L^{2}}^{2}=8\left\| u(t)\right\| _{\dot{H}^{1}}^{2} -\frac{4(n\sigma_{1}+2b)}{\sigma_{1}+2} \int{|x|^{-b} \left|u(t,x)\right|^{\sigma_{1}+2}dx}\le -c<0,
\end{equation}
which implies that the solution blows up in finite time.

$\cdot$ \textbf{Radially symmetric data.} Let us assume in addition that $u_{0}$ is radially symmetric.
Using Lemma 3.4 of \cite{D18}, we have
\begin{eqnarray}\begin{split}\nonumber
\frac{d^2}{dt^2}V_{\varphi_{R}}(t)\le& 8\left\| u(t)\right\| _{\dot{H}^{1}}^{2} -\frac{4(n\sigma_{1}+2b)}{\sigma_{1}+2} \int{|x|^{-b} \left|u(t,x)\right|^{\sigma_{1}+2}dx}\\
&+O\left(R^{-2}+\epsilon^{-\frac{\sigma}{4-\sigma}}R^{-{\frac{2\left[(n-1)\sigma+2b\right]}
{4-\sigma}}}+\epsilon\left\|u(t)\right\| _{\dot{H}^{1}}^{2}\right),
\end{split}\end{eqnarray}
for any $\epsilon>0$ and any $t$ in the existence time. Taking $\epsilon>0$ small enough and $R>1$ large enough depending on $\epsilon$, it follows from \eqref{GrindEQ__5_18_} that
\[\frac{d^2}{dt^2}V_{\varphi_{R}}(t)\le-c/2<0.\]
This shows that the solution must blow up in finite time.

$\cdot$ \textbf{Cylindrically symmetric data.} Let us assume in addition that $b\ge 4-n$ and $u_{0}\in \Sigma_{n}$.
It follows from (4.12) of \cite{DK21} that
\begin{eqnarray}\begin{split}\nonumber
\frac{d^2}{dt^2}V_{\phi_{R}}(t)\le
8\left\| u(t)\right\| _{\dot{H}^{1}}^{2}&-\frac{4(n\sigma_{1}+2b)}{\sigma_{1}+2} \int{|x|^{-b} \left|u(t,x)\right|^{\sigma_{1}+2}dx}\\
&+CR^{-2}+\left\{\begin{array}{l} {CR^{-\frac{n-2}{2}-b}\left\|u(t)\right\|_{\dot{H}^1}^{2},~\textrm{if}\;\sigma_{1}=2,}
\\{CR^{-\frac{(n-2)\sigma_{1}}{4}-b}\left\|u(t)\right\|_{\dot{H}^1}^{2}
+CR^{-\frac{(n-2)\sigma_{1}}{4}-b},~\textrm{if}\;\sigma_{1}<2,} \end{array}\right.
\end{split}\end{eqnarray}
for any $t\in (-T_{*},\;T^{*})$. Taking $R>1$ large enough, \eqref{GrindEQ__5_18_} shows that
\[\frac{d^2}{dt^2}V_{\phi_{R}}(t)\le-c/2<0,\]
which implies that the solution blows up in finite time.
This completes the proof.
\end{proof}

%%%%%%%%%%%%%%%%%%%%%%%%%%%%%%%%%%%%%%%%%%%%%%%%%%%%%%%%%%%%%%%%%%%%%%%%%%%%%

%\section*{References}

\bigskip
E-mail address: cioc12@ryongnamsan.edu.kp (J.An); jm.kim0211@ryongnamsan.edu.kp (J. Kim)
\end{document}